\documentclass[12pt,american,english]{article}
\usepackage[T1]{fontenc}
\usepackage[latin9]{inputenc}
\usepackage{geometry}
\usepackage{xcolor}
\usepackage{babel}
\usepackage{amsmath}
\usepackage{amsthm}
\usepackage{amssymb}
\usepackage{stmaryrd}
\usepackage{graphicx}
\usepackage{xargs}[2008/03/08]
\usepackage[unicode=true,pdfusetitle,
 bookmarks=true,bookmarksnumbered=false,bookmarksopen=false,
 breaklinks=false,pdfborder={0 0 1},backref=false,colorlinks=false]
 {hyperref}

\makeatletter
\numberwithin{equation}{section}
\theoremstyle{plain}
\newtheorem{thm}{\protect\theoremname}[section]
\theoremstyle{remark}
\newtheorem*{acknowledgement*}{\protect\acknowledgementname}
\theoremstyle{definition}
\newtheorem{defn}[thm]{\protect\definitionname}
\theoremstyle{plain}
\newtheorem{lem}[thm]{\protect\lemmaname}
\theoremstyle{definition}
\newtheorem{condition}[thm]{\protect\conditionname}
\theoremstyle{plain}
\newtheorem{prop}[thm]{\protect\propositionname}
\theoremstyle{remark}
\newtheorem{rem}[thm]{\protect\remarkname}
\theoremstyle{remark}
\newtheorem{notation}[thm]{\protect\notationname}
\theoremstyle{plain}
\newtheorem{cor}[thm]{\protect\corollaryname}

\@ifundefined{date}{}{\date{}}

\theoremstyle{plain}

\usepackage{amssymb}

\makeatother

\addto\captionsamerican{\renewcommand{\acknowledgementname}{Acknowledgement}}
\addto\captionsamerican{\renewcommand{\conditionname}{Condition}}
\addto\captionsamerican{\renewcommand{\corollaryname}{Corollary}}
\addto\captionsamerican{\renewcommand{\definitionname}{Definition}}
\addto\captionsamerican{\renewcommand{\lemmaname}{Lemma}}
\addto\captionsamerican{\renewcommand{\notationname}{Notation}}
\addto\captionsamerican{\renewcommand{\propositionname}{Proposition}}
\addto\captionsamerican{\renewcommand{\remarkname}{Remark}}
\addto\captionsamerican{\renewcommand{\theoremname}{Theorem}}
\addto\captionsenglish{\renewcommand{\acknowledgementname}{Acknowledgement}}
\addto\captionsenglish{\renewcommand{\conditionname}{Condition}}
\addto\captionsenglish{\renewcommand{\corollaryname}{Corollary}}
\addto\captionsenglish{\renewcommand{\definitionname}{Definition}}
\addto\captionsenglish{\renewcommand{\lemmaname}{Lemma}}
\addto\captionsenglish{\renewcommand{\notationname}{Notation}}
\addto\captionsenglish{\renewcommand{\propositionname}{Proposition}}
\addto\captionsenglish{\renewcommand{\remarkname}{Remark}}
\addto\captionsenglish{\renewcommand{\theoremname}{Theorem}}
\providecommand{\acknowledgementname}{Acknowledgement}
\providecommand{\conditionname}{Condition}
\providecommand{\corollaryname}{Corollary}
\providecommand{\definitionname}{Definition}
\providecommand{\lemmaname}{Lemma}
\providecommand{\notationname}{Notation}
\providecommand{\propositionname}{Proposition}
\providecommand{\remarkname}{Remark}
\providecommand{\theoremname}{Theorem}

\begin{document}
\selectlanguage{american}%
\global\long\def\RR{\mathbb{R}}%
\global\long\def\CC{\mathbb{C}}%
\global\long\def\HH{\mathbb{H}}%
\global\long\def\NN{\mathbb{N}}%
\global\long\def\ZZ{\mathbb{Z}}%
\global\long\def\QQ{\mathbb{Q}}%

\global\long\def\l{\ell}%

\global\long\def\z{\zeta}%
\global\long\def\e{\epsilon}%
\global\long\def\a{\alpha}%
\global\long\def\b{\beta}%
\global\long\def\ga{\gamma}%
\global\long\def\ph{\varphi}%
\global\long\def\om{\omega}%
\global\long\def\lm{\lambda}%
\global\long\def\dl{\delta}%
\global\long\def\t{\theta}%
\global\long\def\s{\sigma}%
\foreignlanguage{english}{}
\global\long\def\little{\varepsilon}%

\selectlanguage{english}%

\global\long\def\exd#1#2{\underset{#2}{\underbrace{#1}}}%
\global\long\def\exup#1#2{\overset{#2}{\overbrace{#1}}}%

\global\long\def\leb#1{\operatorname{Leb(#1)}}%

\global\long\def\haar{\operatorname{Haar}}%

\global\long\def\norm#1{\left\Vert #1\right\Vert }%

\global\long\def\Onevec{\underline{1}}%
\global\long\def\One{\mathbf{1}}%

\global\long\def\porsmall{\prec}%
\global\long\def\porbig{\succ}%

\global\long\def\diffeo{\simeq}%

\global\long\def\lra{\longrightarrow}%
\global\long\def\del{\partial}%
\global\long\def\dprime{\prime\prime}%

\global\long\def\manifold{\mathcal{M}}%
\global\long\def\base{\mathbf{B}}%

\global\long\def\transpose{\mbox{t}}%

\global\long\def\norm#1{\left\Vert #1\right\Vert }%
\global\long\def\lilnorm#1{\Vert#1\Vert}%

\global\long\def\brac#1{(#1)}%
\global\long\def\vbrac#1{|#1|}%
\global\long\def\sbrac#1{[#1]}%
\global\long\def\dbrac#1{\langle#1\rangle}%
\global\long\def\cbrac#1{\{#1\}}%

\global\long\def\liev{\mathfrak{h}}%

\global\long\def\gl#1{\operatorname{GL}_{#1}}%

\global\long\def\sl#1{\operatorname{SL}_{#1}}%

\global\long\def\so#1{\operatorname{SO}_{#1}}%

\global\long\def\ort#1{\operatorname{O}_{#1}}%

\global\long\def\pgl#1{\operatorname{PGL}_{#1}}%

\global\long\def\po#1{\operatorname{PO}_{#1}}%

\global\long\def\Lat{\Gamma}%
\global\long\def\disgrp{\Gamma}%
\global\long\def\wc{Q}%

\global\long\def\prim{\operatorname{prim}}%

\global\long\def\conv{\operatorname{conv}}%

\global\long\def\interior#1{\operatorname{int}\left(#1\right)}%

\global\long\def\dist#1{\operatorname{dist}\brac{#1}}%

\global\long\def\diag#1{\operatorname{diag}(#1)}%

\global\long\def\rank#1{\operatorname{rank}\brac{#1}}%

\global\long\def\lie#1{\operatorname{Lie}\brac{#1}}%

\global\long\def\vol{\mbox{vol}}%

\global\long\def\covol#1{\operatorname{covol}\brac{#1}}%

\global\long\def\sym{\mbox{Sym}}%

\global\long\def\sp#1#2{\mbox{span}_{#1}(#2)}%

\global\long\def\id{\operatorname{id}}%

\global\long\def\idmat#1{\operatorname{I}_{#1}}%

\global\long\def\comp#1{\operatorname{#1}^{c}}%

\global\long\def\trunc#1#2{#1^{#2}}%

\global\long\def\leng{\ell}%

\global\long\def\nbhd#1#2{\mathcal{O}_{#1}^{#2}}%

\global\long\def\symfund#1{F_{#1}}%
\global\long\def\groupfund#1{\widetilde{F_{#1}}}%

\global\long\def\symfundgen#1{\mathbf{F}_{#1}}%
\global\long\def\groupfundgen#1{\widetilde{\mathbf{F}_{#1}}}%

\global\long\def\symfundrec#1#2{F_{#1}^{\left(#2\right)}}%
\global\long\def\groupfundrec#1#2{\widetilde{F_{#1}}^{\left(#2\right)}}%

\global\long\def\Gset{\mathcal{B}}%
\global\long\def\latticeset{\mathcal{S}}%
\global\long\def\pairset{\Xi}%
\global\long\def\symset{\mathcal{E}}%
\global\long\def\groupset{\widetilde{\symset}}%
\global\long\def\sphereset{\Phi}%

\global\long\def\funddom{\Omega}%

\global\long\def\parby#1#2{#1_{#2}}%

\global\long\def\bypar#1#2{_{#2}#1}%

\global\long\def\byparby#1#2{_{#2}#1_{#2}}%

\global\long\def\ball#1{B_{#1}}%

\global\long\def\GIcomp{S}%
\global\long\def\cube{\square}%

\global\long\def\sphere#1{\mathbb{S}^{#1}}%

\global\long\def\gras#1{\operatorname{Gr}(#1)}%

\global\long\def\grass#1{\operatorname{Gr}^{0}(#1)}%

\global\long\def\latspace#1{\mathcal{L}_{#1}}%

\global\long\def\unilatspace#1{\mathcal{U}_{#1}}%

\global\long\def\shapespace#1{\mathcal{X}_{#1}}%

\global\long\def\flagspace#1{\mathcal{P}_{#1}}%

\global\long\def\flag{\mathbf{F}}%

\global\long\def\svec{\underline{s}}%
\global\long\def\Svec{\underline{S}}%

\global\long\def\sumS{\mathbf{S}}%
\global\long\def\sums{\mathbf{s}}%

\global\long\def\wvec{\underline{w}}%
\global\long\def\Wvec{\underline{W}}%

\global\long\def\sumW{\mathbf{W}}%
\global\long\def\sumw{\mathbf{w}}%

\global\long\def\roundo{r}%

\global\long\def\errexp{\tau}%

\global\long\def\ac{\operatorname{ac}}%

\global\long\def\mass#1{\Vert#1\Vert}%

\global\long\def\sgn#1{\iota(#1)}%
\global\long\def\ind{h}%

\global\long\def\lat{\Lambda}%
\global\long\def\latfull{\Delta}%
\global\long\def\qlat{L}%
\global\long\def\Flat#1#2{#1^{(#2)}}%

\global\long\def\dvec{\underline{d}}%

\global\long\def\cov{X}%

\global\long\def\shape#1{\operatorname{shape}\left(#1\right)}%

\global\long\def\pair#1{\operatorname{pair}\left(#1\right)}%

\global\long\def\unilat#1{\left[#1\right]}%

\global\long\def\unisimlat#1{\left\llbracket #1\right\rrbracket }%

\global\long\def\dir#1{\widehat{#1}}%

\global\long\def\perpen#1{#1^{\perp}}%
\global\long\def\perpeng#1#2{#1^{\perp,#2}}%

\global\long\def\factor#1{#1^{\pi}}%
 
\global\long\def\factorg#1#2{#1^{\pi,#2}}%

\global\long\def\dual#1{#1^{*}}%

\global\long\def\latlast#1#2{#1^{\underleftarrow{#2}}}%

\global\long\def\based{\mathbf{D}}%
\global\long\def\basec{\mathbf{C}}%

\global\long\def\ii{\hat{i}}%
\global\long\def\jj{i}%
\global\long\def\dd{r}%

\global\long\def\iip{\hat{y}}%
\global\long\def\jjp{y}%

\title{Counting flags of primitive lattices}
\author{Tal Horesh\thanks{IST Austria, \texttt{tal.horesh@ist.ac.at}.} \and
Yakov Karasik\thanks{Technion, Israel, \texttt{theyakov@gmail.com}.}}
\maketitle
\begin{abstract}
We count flags of primitive lattices, which are objects of the form
$\left\{ 0\right\} =\Flat{\lat}0<\Flat{\lat}1<\cdots<\Flat{\lat}{\leng}=\ZZ^{n}$,
where every $\Flat{\lat}i$ is a primitive lattice in $\ZZ^{n}$.
The counting is with respect to two different natural height functions,
allowing us to give a new proof of the Manin conjecture for flag varieties
over the rational numbers. We deduce the equidistribution of rational
points in flag varieties, as well as the equidistribution of the shapes
of the successive quotient lattices $\Flat{\lat}i/\Flat{\lat}{i-1}$.
In doing so, we generalize previous work of Schmidt, as well as our
own, on counting primitive lattices of rank $d<n$.
\end{abstract}
\setcounter{tocdepth}{1}\tableofcontents{}

\section{Introduction\label{sec: Introduction}}

Let $n>1$ and let $\underline{d}=\brac{d_{1},\ldots d_{\leng}}$
be a partition of $n$, namely an $\leng$-tuple of (strictly) positive
integers such that $d_{1}+\cdots+d_{\leng}=n$. Consider a flag of
subspaces of $\RR^{n}$,
\begin{equation}
\flag=\:(\,\left\{ 0\right\} =V_{0}<V_{1}<\cdots<V_{\leng}=\RR^{n}\,),\label{eq: flag}
\end{equation}
with $\dim(V_{i})=d_{1}+\cdots+d_{i}$ for all $1\leq1\leq n$. 
If all the subspaces $V_{i}$ are rational (that is, have a basis
consisting of rational vectors), then each contains a unique primitive
lattice of rank $\dim(V_{i})$,
\[
\Flat{\lat}i=V_{i}\cap\ZZ^{n},
\]
and we obtain a flag of primitive lattices
\begin{equation}
\flag\brac{\ZZ}=\;(\,\left\{ 0\right\} =\Flat{\lat}0<\Flat{\lat}1<\cdots<\Flat{\lat}{\leng}=\ZZ^{n}\,).\label{eq: lat flag}
\end{equation}
The aim of this paper is to extend known counting and equidistribution
results from primitive lattices to flags of primitive lattices. Schmidt
\cite[Thm.~1]{Schmidt_68} was the first to prove a counting result
for primitive lattices of rank $1\leq d\leq n$ in $\RR^{n}$, showing
that the number of primitive lattices of rank $d$ with covolume up
to $X$ is 
\begin{equation}
c_{d,n}X^{n}+O\brac{X^{n-\max\cbrac{\frac{1}{d},\frac{1}{n-d}}}},\label{eq: Schmidt}
\end{equation}
where the covolume of a lattice is the volume of a fundamental parallelepiped
for the lattice in the linear space it spans and 
\begin{equation}
c_{d,n}=\frac{1}{n}{n \choose d}\frac{\zeta(2)\cdots\zeta(d)}{\zeta(n-d+1)\cdots\zeta(n)}\frac{\mathfrak{V}(n-d+1)\cdots\mathfrak{V}(n)}{\mathfrak{V}(1)\cdots\mathfrak{V}(d)};\label{eq: constant Schmidt}
\end{equation}
here $\zeta$ is the Riemann Zeta function and $\mathfrak{V}(i)$
is the Lebesgue volume of the a ball in $\RR^{i}$. (This result was
generalized to general number fields by Thunder \cite[Thm.~1]{Thunder92},
who also proved a counting result \cite[Thm.~3]{Thunder93} for primitive
$d$-lattices that do not intersect a certain $\brac{n-d}$-dimensional
subspace. The error term in (\ref{eq: Schmidt}) was distilled by
Kim \cite[Thm.~1.3]{Kim19}). Later, Schmidt \cite[Thm.~2]{Schmidt_98}
refined (\ref{eq: Schmidt}) so that it also takes into account the
\emph{shape} of the lattices, where the shape of a lattice in $\RR^{n}$
is its equivalence class modulo rotation in $\RR^{n}$ and rescaling
by a positive scalar. The space of shapes of rank $d$ lattices is
denoted by $\shapespace d$ (to be defined explicitly in Section \ref{sec: Flags});
it is not compact, but it admits a natural uniform probability measure,
$\vol_{\shapespace d}^{1}$. Schmidt showed that given a Jordan measurable
subset $\symset\subset\shapespace d$, the number of primitive lattices
with covolume up to $X$ and shape inside $\symset$ is\setlength{\abovedisplayskip}{3pt} 
\setlength{\belowdisplayskip}{6pt}
\[
\sim c_{d,n}\cdot\vol_{\shapespace d}^{1}\brac{\symset}X^{n}.
\]
Since the subsets $\symset$ are general enough, this counting can
be read as an equidistribution statement, namely that the shapes of
primitive lattices equidistribute in $\shapespace d$ as their covolume
tends to infinity. 

Dynamical techniques opened the door to equidistribution theorems
that do not follow from (nor imply) counting statements, but with
the advantage of considering lattices of covolume exactly $X$ (as
apposed to at most $X$). Primarily, the focus was on the case $d=n-1$,
namely on primitive lattices that lie in hyperplanes defined by being
orthogonal to primitive vectors. Such equidistribution results were
established by Aka, Einsiedler and Shapira \cite{AES_16B,AES_16A},
Einsiedler, Mozes, Shah and Shapira \cite{EMSS_16}, and (with a bound
on the rate of convergence) by Einsiedler, R�hr and Wirth \cite{ERW17}.
In fact, these equidistribution results were joint for the shapes
of the primitive lattices in $\shapespace{n-1}$, and the projections
of the primitive vectors orthogonal to these lattices to the unit
sphere in $\RR^{n}$ \textemdash{} in other words, for shapes of primitive
$\brac{n-1}$-lattices, and their \emph{directions}. For general $d$,
the direction of a $d$-lattice $\lat$ is the real linear space that
it spans,\setlength{\abovedisplayskip}{0pt} 
\setlength{\belowdisplayskip}{6pt}
\[
V_{\lat}=\sp{\RR}{\lat},
\]
lying in the Grassmannian $\grass{d,n}$ of $d$-dimesional subspaces
of $\RR^{n}$. In \cite[Thm.~1.2]{Schmidt_15}, Schmidt showed that
for quite restricted types of sets $\symset\subset\shapespace d$
and $\sphereset\subseteq\grass{d,n}$, the number of primitive $d$-lattices
with shapes in $\symset$ and directions in $\sphereset$ is \setlength{\abovedisplayskip}{6pt} 
\setlength{\belowdisplayskip}{6pt}
\[
c_{d,n}\cdot\vol_{\shapespace d}^{1}\brac{\symset}\vol_{\grass{d,n}}^{1}\brac{\sphereset}X^{n}+O\brac{X^{n-\frac{1}{d}}\cdot\log^{d-1}X},
\]
where $\vol_{\grass{d,n}}^{1}$ is the uniform probability measure
on $\grass{d,n}$. In \cite{HK_dlattices}, we were able to extend
the above result to subsets $\symset$ and $\sphereset$ that were
general enough to conclude equidistribution, as well as to consider
the orthogonal lattices \setlength{\abovedisplayskip}{6pt} 
\setlength{\belowdisplayskip}{6pt}
\[
\perpen{\lat}=\ZZ^{n}\cap\perpen{V_{\lat}}
\]
to primitive lattices $\lat$, where $\perpen{V_{\lat}}$ is the orthogonal
complement of $V_{\lat}$ in $\RR^{n}$. The consideration of the
orthogonal lattices proved to be crucial in an application to the
study of rational points on Grassmannians, described below. Finally,
Aka, Musso and Wieser \cite{AMW21} have extended the aforementioned
equidistribution results on shapes of primitive lattices with covolume
$X$, from rank $n-1$ to a general rank. 

Our goal in the present paper is to generalize the counting result
in \cite{HK_dlattices} from primitive lattices to flags of such.
To this end, for $\underline{d}=\brac{d_{1},\ldots d_{\leng}}$ as
above, we let 
\begin{equation}
\grass{\dvec,n}=\text{space of \ensuremath{\dvec}-flags in \ensuremath{\RR^{n}}}\label{eq: flag grassmannian (non-ortd)}
\end{equation}
(a $\dvec$-flag is the object defined in (\ref{eq: flag})), and\setlength{\abovedisplayskip}{3pt} 
\setlength{\belowdisplayskip}{3pt}
\begin{equation}
\shapespace{\dvec}=\prod_{i=1}^{\leng}\shapespace{d_{i}}\,.\label{eq: flag shape space}
\end{equation}
Notice that one cannot expect equidistribution of the projections
of $\Flat{\lat}i$ to $\shapespace{d_{1}+\cdots+d_{i}}$ or to $\grass{d_{i},n}$
jointly for all $i=1,\ldots,n$, since the relation of inclusion between
the $\Flat{\lat}i$'s implies dependence. Instead, we consider the
successive quotients:
\[
L_{1}=\Flat{\lat}1/\Flat{\lat}0\,,\ldots,\,L_{\leng}=\Flat{\lat}{\leng}/\Flat{\lat}{\leng-1},
\]
where we note that $\rank{L_{i}}=d_{i}$ for all $1\leq i\leq\leng$.
Let 
\[
\shape{\flag\brac{\ZZ}}=\left(\shape{L_{1}},\ldots,\shape{L_{\leng}}\right)\in\shapespace{\dvec}\,.
\]
(As explained in Section \ref{sec: Flags}, the quotients $\qlat_{i}$
are isometric to concrete lattices in $\RR^{n}$, so their shapes
are well defined). Our first counting result is with respect to the
height function 
\[
H_{\infty}(\flag\brac{\ZZ})=\max\cbrac{\covol{\Flat{\lat}1},\ldots,\covol{\Flat{\lat}{\leng}}}.
\]
A subset of an orbifold is called \emph{boundary controllable} \cite[Def.~1.2]{HK_WellRoundedness}
if its boundary satisfies a standard regularity condition (Definition
\ref{def: BCS}).
\begin{thm}
\label{thm: Counitng with supremum}Let $\symset\subseteq\shapespace{\dvec}$
and $\sphereset\subseteq\grass{\dvec,n}$ be boundary controllable.
The number of primitive lattice $\dvec$-flags $\flag\brac{\ZZ}$
with $H_{\infty}(\flag\brac{\ZZ})\leq X$, $\flag\in\sphereset$ and
$\shape{\flag\brac{\ZZ}}\in\symset$ is 
\[
c_{\dvec,n}\cdot\vol_{\shapespace{\dvec}}^{1}\brac{\symset}\vol_{\grass{\dvec,n}}^{1}\brac{\sphereset}\cdot X^{2n-d_{1}-d_{\leng}}+O_{\e}\brac{X^{\brac{2n-d_{1}-d_{\leng}}\left(1-\frac{1}{16n^{2}}+\e\right)}}
\]
for all $\e>0$, where 
\begin{equation}
c_{\dvec,n}=\frac{1}{2^{\leng-1}}\frac{1}{\prod_{i=1}^{\leng-1}(d_{i}+d_{i+1})}{n \choose d_{1},\ldots,d_{\leng}}\frac{\prod_{i=1}^{\leng-1}\prod_{j=2}^{d_{i}}\zeta\brac j}{\prod_{i=2}^{d_{\leng}}\zeta\left(i\right)}\frac{\prod_{i=d_{\leng}+1}^{n}\mathfrak{V}(i)}{\prod_{i=1}^{\leng-1}\prod_{j=1}^{d_{i}}\mathfrak{V}(j)}.\label{eq: constant flags}
\end{equation}
\end{thm}

Notice that when $\leng=2$, the lattice flag $\flag(\ZZ)$ is in
fact a single primitive lattice $\lat<\ZZ^{n}$, and then the constant
$c_{\dvec,n}$ coincides with Schmidt's constant (\ref{eq: constant Schmidt}).
Indeed, returning to Schmidt's result on primitive lattices, the one
to one correspondence between primitive $d$-lattices in $\ZZ^{n}$
and rational $d$-dimensional subspaces of $\RR^{n}$ ($V\mapsto V\cap\ZZ^{n}$
and $V_{\lat}\ensuremath{\mapsfrom}\lat$) means that the primitive
$d$-lattices are in fact the rational points on the projective variety
$\grass{d,n}$. Since the anticanonical height function on this variety
is
\[
H_{\ac}(V_{\lat})=\covol{\lat}^{n},
\]
then (\ref{eq: Schmidt}) can be read as one on counting rational
points on the Grassmannian w.r.t.\ the anticanonical height function,
and in particular it confirms Manin's Conjecture \cite{FMT89,P95}
for this variety. The anticanonical height function on the flag variety
$\grass{\dvec,n}$ (whose elements are $\dvec$-flags of the form
(\ref{eq: flag}), and whose rational points are primitive lattice
$\dvec$-flags as in (\ref{eq: lat flag})) is 

\[
H_{\ac}(\flag\brac{\ZZ})=\prod_{i=0}^{\leng-1}\covol{\Flat{\lat}i}^{d_{i}+d_{i+1}}
\]
\cite{Papa83,Thunder93}, and it is known by the work of Franke, Manin
and Tschinkel \cite[Cor.~from Thm.~5]{FMT89} that flag varieties
also satisfy Manin's conjecture (see also \cite[Thm.~5]{Thunder93}
and \cite[Cor.~1.3]{Kim19} for the special case of flags  in which
$\Flat{\lat}1$ intersects trivially a given subspace). However, just
as (\ref{eq: Schmidt}) can be refined to include the shapes and directions
of primitive lattices, so can the counting of primitive lattice flags.
Our second result is on counting primitive lattice flags with respect
to the height function $H_{\ac}$, and with consideration of shapes
and directions.
\begin{thm}
\label{thm: counting anti canonical}Let $\symset\subseteq\shapespace{\dvec}$
and $\sphereset\subseteq\grass{\dvec,n}$ be boundary controllable.
The number of primitive lattice $\dvec$-flags $\flag\brac{\ZZ}$
with $H_{\ac}(\flag\brac{\ZZ})\leq X$, $\flag\in\sphereset$ and
$\shape{\flag\brac{\ZZ}}\in\symset$ is 
\[
c_{\dvec,n}\cdot\vol_{\shapespace{\dvec}}^{1}\brac{\symset}\vol_{\grass{\dvec,n}}^{1}\brac{\sphereset}\cdot X\sum_{j=0}^{\leng-2}\frac{\brac{-1}^{\leng-2-j}}{j!}\brac{\log X}^{j}+O_{\e}\brac{X^{\left(1-\frac{1}{16n^{2}}+\e\right)}}
\]
for all $\e>0$, where $c_{\dvec,n}$ is as in (\ref{eq: constant flags}).
\end{thm}

The refinement of Franke, Manin and Tschinkel's result suggested in
Theorem \ref{thm: counting anti canonical} could prove useful in
further study of rational points on flag varieties: Browning, the
first author and Wilsch \cite{BHW_grassmannians} built on \cite{HK_dlattices}
to establish the \emph{freeness} variant of Manin's conjecture, proposed
by Peyre \cite{P17,P18}, for Grassmannians. We expect that Theorem
\ref{thm: counting anti canonical} could be used to extend the results
in \cite{BHW_grassmannians} from Grassmannians to more general flag
varieties.

\paragraph{Organization of the paper}

In Section \ref{sec: Flags}, we provide some background on lattices
and define a \emph{space of primitive unimodular flags}, which generalizes
the concept of the space of unimodular lattices $\sl n\brac{\RR}/\sl n\brac{\ZZ}$.
In Section \ref{sec: RI coordinates}, we define a refinement of the
Iwasawa coordinates on $\sl n(\RR)$ that is suiting for studying
this space, as well as the spaces $\shapespace{\dvec}$ and $\grass{\dvec,n}$.
The analysis of the $\shapespace{\dvec}$, $\grass{\dvec,n}$ and
the space of primitive unimodular flags is completed in Section \ref{sec: Spread Models},
including the measures $\vol$ whose normalizations $\vol^{1}$ to
probability measures appear in Theorems \ref{thm: Counitng with supremum}
and \ref{thm: counting anti canonical}. In Section \ref{sec: General Thm},
we state the more general Theorem \ref{thm: general thm} for counting
lattice flags $\flag\brac{\ZZ}$, this time with respect to their
projections to the space of primitive unimodular flags, and prove
Theorems \ref{thm: Counitng with supremum} and \ref{thm: counting anti canonical}
based on it. The rest of the paper is devoted to proving Theorem \ref{thm: general thm}.
In Section \ref{sec: Integral matrices representing primitive vectors},
we translate the problem of counting the flags $\flag\brac{\ZZ}$
into a problem of counting the points of the integral lattice $\sl n\brac{\ZZ}$
in carefully designed subsets of $\sl n(\RR)$, whose volumes are
computed in Section \ref{sec: volumes}. These subsets are not compact
\textendash{} we split them into compact subsets that contain ``most
of the mass'' (and most lattice points), which we handle in Section
\ref{sec: counting with GN}, and to their non-compact complements,
which we handle in Section \ref{sec: Counting the cusp}. 

\begin{acknowledgement*}
The initial idea for this work was born while both authors were visiting
IH�S (Institut des Hautes �tudes Scientifiques, France) during the
end of 2018. It was then developed into a paper while both authors
were at IST Austria at the end of 2020 and again at the end of 2021.
During these visits, the support of EPSRC grant EP/P026710/1 is gratefully
acknowledged. The authors want to express their deep gratitude to
Florian Wilsch for suggesting the idea of studying the anticanonical
height function and for many extremely helpful discussions. 
\end{acknowledgement*}

\section{\label{sec: Flags}From lattices to flags of lattices}

Let $V$ be a real vector space of dimension $n$. A \emph{$d$-lattice}
(or, a lattice of rank $d$) $\lat<V$ is the $\ZZ$-span of $1\leq d\leq n$
linearly independent elements in $V$. When $d=n$, we say that $\lat$
is a \emph{full} lattice. Recall that $V_{\lat}<V$ is the real vector
space spanned by $\lat$, and that the covolume of $\lat$, $\covol{\lat}$,
is the volume of a fundamental parallelepiped of $\lat$ in $V_{\lat}$.
Given a basis $\base$ for $\lat$, the covolume of $\lat$ is $\brac{|\det\brac{\base^{\transpose}\base}|}^{1/2}$.
For technical reasons, we will regard our lattices $\lat$, and accordingly
the linear spaces that they span $V_{\lat}$, as \emph{oriented} lattices
(resp.\ subspaces), meaning that they are equipped with a choice
of orientation. We say that a lattice $\lat$ is \emph{unimodular}
if it is positively oriented and has covolume one. The space of unimodular
$d$-lattices in $\RR^{d}$ is 
\[
\latspace d=\sl d\left(\RR\right)/\sl d\left(\ZZ\right),
\]
and the space of shapes of $d$-lattices is 
\[
\shapespace d=\so d\left(\RR\right)\backslash\sl d\left(\RR\right)/\sl d\left(\ZZ\right)
\]
(recall that the shape of $\lat$ is its equivalence class modulo
rotation and rescaling). Finally, we let
\[
\gras{d,n}=\text{set of oriented \ensuremath{d}-lattices in \ensuremath{\RR^{n}},}
\]
which is a double cover of $\grass{d,n}$. Just as the direction of
a lattice $\lat$ is  the real vector space $V_{\lat}$ that it spans,
the direction of an oriented lattice $\lat$ is  the real oriented
subspace that it spans; we keep the notation $V_{\lat}$. 

Clearly, $\ZZ^{n}$ (with a positive orientation) is a full unimodular
lattice in $\RR^{n}$. Given a lattice of smaller rank $\lat<\ZZ^{n}$,
it is standard to call $\lat$ primitive if $\lat=V_{\lat}\cap\ZZ^{n}$.
This notion naturally extends from $\ZZ^{n}$ to any other full lattice
$\latfull<\RR^{n}$ as follows. 
\begin{defn}
Assume that a $d$-lattice $\lat$ is contained inside a full lattice
$\latfull<\RR^{n}$. We say that $\lat$ is \emph{primitive} inside
$\latfull$ if $\lat=\latfull\cap V_{\lat}$. When $\lat$ is primitive
inside $\ZZ^{n}$, we omit the explicit mentioning of $\ZZ^{n}$,
and just say that $\lat$ is primitive.
\end{defn}

When $\lat$ is primitive in $\latfull$, the quotient $\latfull/\lat$
is a lattice; it is a full lattice in the vector space $V_{\latfull}/V_{\lat}$
and has covolume $\covol{\latfull}/\covol{\lat}$. If, moreover, $V_{\latfull}=\RR^{n}$
and $\lat$ is primitive in $\latfull$, then $\latfull/\lat$ can
be viewed as a lattice in $\RR^{n}$; this is because it is isometric
to the lattice inside $\perpen{V_{\lat}}$ which is obtained by projecting
$\latfull$ orthogonally to $\perpen{V_{\lat}}$. (Such projected
lattices are called \emph{factor lattices}; they are introduced in
\cite{Schmidt_68} and studied in \cite{HK_dlattices}). Moreover,
$\latfull/\lat$ can also be regarded as an oriented lattice, inheriting
the following orientation from the factor lattice: A basis $\basec$
for the factor lattice is positively oriented if $\det(\base|\basec)=1$
for a positively oriented basis $\base$ of $\lat$. Thus, the shape
in $\shapespace{n-d}$ (where $d=\rank{\lat}$) and direction in $\gras{n-d,n}$
of $\latfull/\lat$ are well defined. 

\paragraph*{Flags of lattices.}

A flag of lattices is a finite sequence of lattices in $V$ with strictly
increasing ranks:
\[
\left\{ 0\right\} =\Flat{\lat}0<\Flat{\lat}1<\cdots<\Flat{\lat}{\leng}<V,
\]
where $\Flat{\lat}{\leng}$ is a full lattice in $V$. We say that
a flag of lattices is \emph{primitive} if every $\Flat{\lat}{j-1}$
is primitive in $\Flat{\lat}j$; notice that a flag is primitive if
and only if there exists $g\in\gl n(\RR)$ such that the first $d_{1}$
columns of $g$ span $\Flat{\lat}1$, the first $d_{1}+d_{2}$ columns
span $\Flat{\lat}2$, and so forth, where the whole $n=d_{1}+\cdots+d_{\leng}$
columns span $\Flat{\lat}{\leng}$. We let 
\[
\flag_{g}=(\,\cbrac 0=\Flat{V_{g}}0<\Flat{V_{g}}1<\cdots<\Flat{V_{g}}{\leng}=\RR^{n}\,)
\]
denote the $\dvec$-flag of real vector spaces spanned by $g$, set
the notation $\lat_{g}$ for the lattice spanned by the columns of
$g$, and then let $\flag_{g}\brac{\ZZ}$ denote the (primitive) flag
of lattices spanned by $g$:
\[
\flag_{g}\brac{\ZZ}=\,\flag_{g}\cap\lat_{g}\,=\,(\,\cbrac 0=\Flat{\lat_{g}}0<\Flat{\lat_{g}}1<\cdots<\Flat{\lat_{g}}{\leng}=\lat_{g}\,),
\]
namely $\Flat{\lat_{g}}i=\Flat{V_{g}}i\cap\lat_{g}$ for every $i=1,\ldots,\leng$. 

An \emph{orientation} on a flag of lattices (or on the flag of subspaces
that it spans) is a choice of orientation on a basis for $\Flat{\lat}1$,
then a choice of orientation on $\Flat{\lat}2/\Flat{\lat}1$, and
so on; for the flag spanned by $g$, this means choosing an orientation
separately on every block of columns $\sbrac{d_{1}+\cdots+d_{i-1}+1,d_{1}+\cdots+d_{i}}$,
for $i=1,\ldots,\leng$. All in all, each $\dvec$-flag has $2^{\leng}$
possible orientations, and we say that an oriented flag $\flag_{g}$
is \emph{positive} if $\det(g)>0$ (so, a positive flag has $2^{\leng-1}$
possible orientations). Let\setlength{\abovedisplayskip}{4pt} 
\setlength{\belowdisplayskip}{6pt}
\begin{align*}
\gras{\dvec,n} & =\text{set of positive (oriented) \ensuremath{\dvec}-flags in \ensuremath{\RR^{n}},}\\
 & =\so n(\RR)/\left[\begin{smallmatrix}\so{d_{1}}\left(\RR\right) & \cdots & 0\\
\vdots & \ddots & \vdots\\
0 & \cdots & \so{d_{\leng}}\left(\RR\right)
\end{smallmatrix}\right],
\end{align*}
which is a $2^{\leng-1}$\textendash cover of $\grass{\dvec,n}$.
A primitive oriented flag of lattices is called \emph{unimodular}
if it is positive and all the successive quotients $\Flat{\lat}j/\Flat{\lat}{j-1}$
$1\leq j\leq\leng$ have covolume one; in particular, a unimodular
flag must be primitive (otherwise the quotients would not be lattices).

Notice that if $g\in\gl n(\ZZ)$ then all the lattices in $\flag_{g}\brac{\ZZ}$
are integral (the largest lattice is $\ZZ^{n}$) and so $\flag_{g}\brac{\ZZ}$
recovers $\flag\brac{\ZZ}$ from (\ref{eq: lat flag}); we refer to
such a flag as a \emph{primitive integral flag}. Then Theorems \ref{thm: Counitng with supremum}
and \ref{thm: counting anti canonical} concern counting primitive
integral flags, with consideration of their projections to $\shapespace{\underline{d}}$
((\ref{eq: flag shape space})) and $\grass{\dvec,n}$ ((\ref{eq: flag grassmannian (non-ortd)})).
The spaces $\shapespace{\underline{d}}$ and $\grass{\dvec,n}$ (resp.\
$\gras{\dvec,n}$) parameterize the shapes and directions of flags
(resp.\ oriented flags) of lattices; however, there exists a space
that parameterizes both of these properties. Consider first the space
of all primitive $\dvec$-flags in $\RR^{n}$,\setlength{\abovedisplayskip}{8pt} 
\setlength{\belowdisplayskip}{7pt}
\[
\gl n\left(\RR\right)/\left[\begin{smallmatrix}\gl{d_{1}}\left(\ZZ\right) & \RR & \RR\\
\vdots & \ddots & \RR\\
0 & \cdots & \gl{d_{\leng}}\left(\ZZ\right)
\end{smallmatrix}\right],
\]
which has infinite volume; compare to the infinite-volume space of
full lattices in $\RR^{n}$, $\gl n(\RR)/\gl n(\ZZ)$,  where to
obtain a finite volume space one restricts to the space of full unimodular
lattices, $\latspace n$. Aiming to imitate this construction, we
define the following subgroup of $\sl n(\RR)$\setlength{\abovedisplayskip}{4pt} 
\setlength{\belowdisplayskip}{5pt}
\begin{equation}
A^{\prime}=\left\{ \left[\begin{smallmatrix}\a_{1}^{\frac{1}{d_{1}}}\idmat{d_{1}} & 0 & \cdots & 0\\
0 & \ddots &  & \vdots\\
\vdots &  & \a_{\leng-1}^{\frac{1}{d_{\leng-1}}}\idmat{d_{\leng-1}} & 0\\
0 & \cdots & 0 & (\alpha_{1}\cdots\alpha_{\leng-1}){}^{\frac{1}{d_{\leng}}}\idmat{d_{\leng}}
\end{smallmatrix}\right]:\a_{1},\ldots,\a_{\leng-1}>0\right\} \label{eq: def A'}
\end{equation}
and consider the space
\begin{align*}
\flagspace{\dvec} & =\text{space of unimodular \ensuremath{\dvec}-flags in \ensuremath{\RR^{n}}}\\
 & =\sl n\left(\RR\right)/\left[\begin{smallmatrix}\sl{d_{1}}\left(\ZZ\right) & \RR^{d_{1}\times d_{2}} & \cdots & \RR^{d_{1}\times d_{\leng}}\\
0 & \ddots & \ddots & \vdots\\
\vdots & \ddots & \sl{d_{\leng-1}}\left(\ZZ\right) & \RR^{d_{\leng-1}\times d_{\leng}}\\
0 & \cdots & 0 & \sl{d_{\leng}}\left(\ZZ\right)
\end{smallmatrix}\right]\times A^{\prime}.
\end{align*}
In a similar way to how any positively-oriented full lattice in $\RR^{n}$
can be projected to the space $\latspace n$ of unimodular lattices
by rescaling to covolume one, any primitive positive $\dvec$-flag
of lattices in $\RR^{n}$ can be ``rescaled'' to a unimodular flag
by rescaling the successive quotiens of the flag. More concretely,
if $g=(\base_{1}|\cdots|\base_{\leng})$ is a basis for the flag,
then one rescales each $\base_{j}$ separately to obtain $\basec_{j}$
such that $\brac{|\det\brac{\basec_{j}^{\transpose}\basec_{j}}|}^{1/2}=1$.
This rescaling is exactly the role of modding out by $A^{\prime}$,
and we denote the rescaled $\flag\brac{\ZZ}$ by $\unilat{\flag\brac{\ZZ}}\in\flagspace{\dvec}$.

Throughout the next two sections, we will study the properties of
the space $\flagspace{\dvec}$, and its relation to $\shapespace{\underline{d}}$
and $\gras{\dvec,n}$.

\section{Refined Iwasawa components of $\protect\sl n\left(\protect\RR\right)$\label{sec: RI coordinates}}

As we have pointed out in the previous section, we should think of
the space of primitive unimodular flags $\flagspace{\dvec}$ as some
sort of an analog for the well known space $\latspace n$. Typically
(e.g. \cite[V]{Bekka_Mayer}), to study $\latspace n$, one uses the
Iwasawa decomposition on $\sl n(\RR)$, 
\[
\sl n(\RR)=KAN
\]
where $K=K_{n}=\so n(\RR)$, $A=A_{n}$ is the diagonal subgroup in
$\sl n(\RR)$ and $N=N_{n}$ is the upper unipotent subgroup. It is
also standard to denote $P=P_{n}=A_{n}N_{n}$. To study $\flagspace{\dvec}$
(as well as $\shapespace{\dvec}$ and $\gras{\dvec,n}$), we use a
refinement of the Iwasawa decomposition, which we now turn to define. 

\subsection{Refining the Iwasawa decomposition of $\protect\sl n\left(\protect\RR\right)$}

Consider the following block-diagonal subgroup of $G=G_{n}=\sl n(\RR)$:
\[
G^{\dprime}=\left(\begin{array}{ccc}
\sl{d_{1}}\left(\RR\right)\\
 & \ddots\\
 &  & \sl{d_{\leng}}\left(\RR\right)
\end{array}\right)\cong\prod_{i=1}^{\leng}G_{d_{i}},
\]
and write $G^{\dprime}=K^{\dprime}A^{\dprime}N^{\dprime}$ for the
Iwasawa decomposition of $G^{\dprime}$, where 
\[
\begin{array}{ccc}
K^{\dprime}=K\cap G^{\dprime}, & A^{\dprime}=A\cap G^{\dprime}, & N^{\dprime}=N\cap G^{\dprime}.\end{array}
\]
Then $K^{\dprime},A^{\dprime}$ and $N^{\dprime}$ are also block
diagonal, and
\[
\begin{array}{ccc}
K^{\dprime}\cong\prod_{i=1}^{\leng}K_{d_{i}}, & A^{\dprime}\cong\prod_{i=1}^{\leng}A_{d_{i}}, & N^{\dprime}\cong\prod_{i=1}^{\leng}N_{d_{i}}.\end{array}
\]
Let $P^{\dprime}=A^{\dprime}N^{\dprime}$ and 
\[
\wc=KP^{\dprime};
\]
notice that $\wc$ is \emph{not} a group, but it is a smooth manifold.
To complete the definition of the Refined Iwasawa decomposition, we
define $K^{\prime},A^{\prime},N^{\prime}$ that complete $K^{\dprime},A^{\dprime},N^{\dprime}$
to $K$, $A$ and $N$ respectively. Let\setlength{\abovedisplayskip}{3pt} 
\setlength{\belowdisplayskip}{6pt}
\[
N^{\prime}=\left[\begin{smallmatrix}\idmat{d_{1}} & \RR^{d_{1}\times d_{2}} & \cdots & \RR^{d_{1}\times d_{\leng}}\\
0 & \ddots & \ddots & \vdots\\
\vdots & \ddots & \idmat{d_{\leng-1}} & \RR^{d_{\leng-1}\times d_{\leng}}\\
0 & \cdots & 0 & \idmat{d_{\leng}}
\end{smallmatrix}\right]
\]
and $A^{\prime}$ as in (\ref{eq: def A'}); observe that $N=N^{\dprime}N^{\prime}$,
$A=A^{\dprime}A^{\prime}$, and that $A^{\prime}$  commutes with
$G^{\dprime}$. Fix a transversal $K^{\prime}$ of the diffeomorphism
$K/K^{\dprime}\to\gras{\dvec,n}$, meaning that $K=K^{\prime}K^{\dprime}$
and $\wc=K^{\prime}G^{\dprime}$. We can assume that $K^{\prime}$
satisfies a certain regularity property that is described in Condition
\ref{cond: K'}. Then the RI decomposition is given by\setlength{\abovedisplayskip}{3pt} 
\setlength{\belowdisplayskip}{6pt}

\[
G=K^{\prime}K^{\dprime}A^{\dprime}A^{\prime}N^{\dprime}N^{\prime}=K^{\prime}G^{\dprime}A^{\dprime}A^{\prime}N^{\prime}=\wc A^{\prime}N^{\prime}.
\]

\subsection{\label{subsec: RI Haar measure}Refining the Iwasawa decomposition
of the Haar measure}

It is well known (e.g. \cite[Prop.~8.43]{Knapp}) that a Haar measure
on $\sl n\left(\RR\right)$ can be decomposed according to the Iwasawa
components of $\sl n\left(\RR\right)$. Let us extend this to a Refined
Iwasawa decomposition of the Haar measure on $\sl n\left(\RR\right)$:
on every $\GIcomp$ appearing as a component in the Iwasawa or Refined
Iwasawa decompositions of $\sl n(\RR)$ (e.g.\ $S=N^{\prime}$, $K$,
$\wc$...) we define a Radon measure $\mu_{\GIcomp}$ such that the
Haar measure on $\sl n(\RR)$ (with the corresponding normalization)
is the product of the measures $\mu_{\GIcomp}$ on the components.
Denote by $\mass{\mu}$ the total mass of a finite measure $\mu$.

First of all, let us introduce a parameterization on $A\cong\RR^{n-1}$
and its subgroups $A^{\prime}\cong\RR^{\leng-1}$ and $A^{\dprime}\cong\RR^{n-\leng}$.
An element $a=\diag{a_{1},\ldots,a_{n}}\in A$ will be written as
$a_{\underline{h}}=a_{\left(h_{1},\ldots,h_{n-1}\right)}$ if 
\[
\left(a_{1},a_{2},\ldots,a_{n-1},a_{n}\right)=\brac{e^{-h_{1}/2},e^{\left(h_{1}-h_{2}\right)/2},\ldots,e^{\left(h_{n-2}-h_{n-1}\right)/2},e^{h_{n-1}/2}}.
\]
Accordingly, we write $a^{\dprime}\in A^{\dprime}$ as $a^{\dprime}=a_{\Flat{\underline{s}}1\ldots\Flat{\underline{s}}{\leng}}^{\dprime}$
with $\Flat{\underline{s}}i\in\RR^{d_{i}-1}$ if  $a^{\dprime}=\text{diag}(a_{\Flat{\underline{s}}1},...,a_{\Flat{\underline{s}}{\leng}})$
and $a_{\Flat{\underline{s}}j}\in A_{d_{j}}$. Finally, every element
in $A^{\prime}$ is of the form
\[
a_{\underline{t}}^{\prime}=a_{t_{1},\ldots,t_{\leng-1}}^{\prime}=\diag{e^{\frac{t_{1}}{d_{1}}}\idmat{d_{1}},e^{\frac{t_{2}-t_{1}}{d_{2}}}\idmat{d_{2}},\ldots,e^{\frac{t_{\leng-1}-t_{\leng-2}}{d_{\leng-1}}}\idmat{d_{\leng-1}},e^{-\frac{t_{\leng-1}}{d_{\leng}}}\idmat{d_{\leng}}}.
\]

We know that a Haar measure $\mu_{G_{n}}$ on $\sl n(\RR)$ can be
decomposed into measures on the Iwasawa subgroups as 
\[
d\mu_{G_{n}}=d\mu_{K_{n}}\times\frac{dh_{1}\cdots dh_{n-1}}{e^{h_{1}}\cdots e^{h_{n-1}}}\times d\mu_{N_{n}},
\]
where $\mu_{K_{n}}$ and $\mu_{N_{n}}$ are Haar measures, and each
$d_{h_{j}}$ is the Lebesgue measure on $\RR$. Let us define $\mu_{A_{n}}$
such that 
\[
d\mu_{A_{n}}=\frac{dh_{1}\cdots dh_{n-1}}{e^{h_{1}}\cdots e^{h_{n-1}}},
\]
fix $\mu_{N_{n}}$ to be the pullback of the Lebesgue measure through
any isomorphism $N_{n}\cong\RR^{\brac{{n \choose 2}-\sum_{j=1}^{\leng}d_{j}^{2}}/2}$,
and let $\mu_{K_{n}}$ be the Haar measure on $\so n(\RR)$ satisfying
that 
\begin{equation}
\mass{\mu_{K_{n}}}=\prod_{i=1}^{n-1}\leb{\sphere i},\label{eq: Haar K}
\end{equation}
where $\leb{\sphere i}$ is the Lesbegue measure of the $i$-th dimensional
unit sphere. The motivation for this choice is that, corresponding
to $\sphere{n-1}\cong K_{n}/K_{n-1}$, we have 
\[
\leb{\sphere{n-1}}=\mass{\mu_{K_{n}}}/\mass{\mu_{K_{n-1}}}.
\]
The choice of $\mu_{K_{n}}$ and $\mu_{N_{n}}$ determine a Haar measure
on $G_{n}$,
\[
\mu_{G_{n}}=\mu_{K_{n}}\times\mu_{A_{n}}\times\mu_{N_{n}},
\]
and since $G^{\dprime}\cong\prod G_{d_{j}}$ we let 
\[
\mu_{G^{\dprime}}=\prod\mu_{G_{d_{j}}}.
\]
The measures $\mu_{K^{\dprime}}$, $\mu_{A^{\dprime}}$ and $\mu_{N^{\dprime}}$
are also defined in that manner ($\mu_{K^{\dprime}}=\prod\mu_{K_{d_{j}}}$
etc.). They thus determine unique $\mu_{K^{\prime}}$, $\mu_{A^{\prime}}$
and $\mu_{N^{\prime}}$ such that 
\[
\mu_{K}=\mu_{K^{\dprime}}\times\mu_{K^{\prime}},\;\:\mu_{A}=\mu_{A^{\dprime}}\times\mu_{A^{\prime}},\;\:\mu_{N}=\mu_{N^{\dprime}}\times\mu_{N^{\prime}};
\]
indeed, $\mu_{N^{\prime}}$ is again a pullback of the Lebesgue measure
on $\RR^{\dim N^{\prime}}$,\setlength{\abovedisplayskip}{4pt} 
\setlength{\belowdisplayskip}{4pt}
\begin{equation}
d\mu_{A^{\prime}}=\prod_{j=1}^{\leng-1}e^{\brac{d_{j}+d_{j+1}}t_{j}}dt_{j},\label{eq: A' measure}
\end{equation}
and $\mu_{K^{\prime}}$ is the pullback of a $K$-invariant Radon
measure on $K/K^{\dprime}$ normalized such that $\mu_{K}=\mu_{K^{\dprime}}\times\mu_{K^{\prime}}$.
Finally, notice that $\wc$ is diffeomorphic to the group $K\times P^{\dprime}$;
hence, we equip it with the measure 
\begin{equation}
\mu_{\wc}=\mu_{K}\times\mu_{P^{\dprime}}=\mu_{K^{\prime}}\times\mu_{G^{\dprime}},\label{eq: measure Q}
\end{equation}
which is clearly invariant under the $K\times P^{\dprime}$ acting
by $\brac{g,h}\cdot q=\brac{g,h}\cdot kp^{\dprime}=gkp^{\dprime}h^{\transpose}$.
We now have that
\begin{align}
\begin{array}{c}
\mu_{G}=\mu_{K^{\prime}}\times\mu_{K^{\dprime}}\times\mu_{A^{\dprime}}\times\mu_{A^{\prime}}\times\mu_{N^{\dprime}}\times\mu_{N^{\prime}}\\
=\mu_{K^{\prime}}\times\mu_{G^{\dprime}}\times\mu_{A^{\prime}}\times\mu_{N^{\prime}}=\mu_{Q}\times\mu_{A^{\prime}}\times\mu_{N^{\prime}}
\end{array} & .\label{eq: Haar measure on G}
\end{align}

\section{\label{sec: Spread Models}Relation between the Refined Iwasawa components
and spaces of flags}

\begin{figure}
\begin{centering}
\includegraphics[scale=0.4]{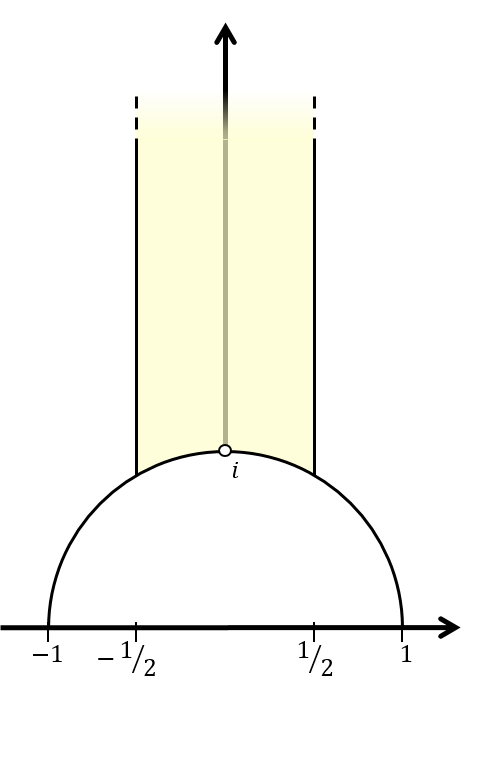}
\par\end{centering}
\caption{\label{fig: fund dom SL(2,Z)}Fundamental domain $\protect\symfund 2$
for $\protect\sl 2\left(\protect\ZZ\right)$ in $P_{2}$ (the hyperbolic
upper half plane).}
\end{figure}

Let us continue the analysis of the spaces $\shapespace{\underline{d}}$,
$\flagspace{\dvec}$ and $\gras{\dvec,n}$ using the Refined Iwasawa
decomposition. Our primary goal is to show how these spaces interact
and to define measures on them. Let us begin with $\gras{\dvec,n}$.
Recalling (\ref{eq: Haar K}) and the fact that $\leb{\sphere{i-1}}=i\cdot\mathfrak{V}(i)$
we have that
\begin{equation}
\mass{\mu_{K_{n}}}=\prod_{i=1}^{n-1}\leb{\sphere i}=\prod_{i=2}^{n}\leb{\sphere{i-1}}=\frac{1}{2}\prod_{i=1}^{n}\leb{\sphere{i-1}}=\frac{1}{2}\prod_{i=1}^{n}i\cdot\mathfrak{V}(i);\label{eq: Kaar K (2)}
\end{equation}
As $\gras{\dvec,n}\diffeo K_{n}/\prod_{j=1}^{\leng}K_{d_{j}}$ (we
use the notation $\simeq$ to indicate a diffeomorphism), we let $\vol_{\gras{\dvec,n}}$
be the unique $K_{n}$-invariant measure on $\gras{\dvec,n}$ normalized
such that 
\begin{align*}
\mass{\vol_{\gras{\dvec,n}}} & =\mass{\mu_{K_{n}}}/\prod_{j=1}^{\leng}\mass{\mu_{K_{d_{j}}}}.
\end{align*}
By (\ref{eq: Kaar K (2)}) this equals 
\[
=\frac{\frac{1}{2}\prod_{i=1}^{n}i\cdot\mathfrak{V}(i)}{\frac{1}{2^{\leng}}\prod_{i=1}^{\leng}\prod_{j=1}^{d_{i}}\prod_{i=1}^{n}j\cdot\mathfrak{V}(j)}=2^{\leng-1}\frac{n!}{d_{1}!\cdots d_{\leng}!}\frac{\prod_{i=1}^{n}\mathfrak{V}(i)}{\prod_{i=1}^{\leng}\prod_{j=1}^{d_{i}}\mathfrak{V}(j)}.
\]
Naturally, we let $\vol_{\grass{\dvec,n}}$ the $K_{n}$-invariant
measure such that
\begin{align}
\mass{\vol_{\grass{\dvec,n}}} & =\frac{1}{2^{\leng-1}}\mass{\vol_{\gras{\dvec,n}}}=\frac{n!}{d_{1}!\cdots d_{\leng}!}\frac{\prod_{i=1}^{n}\mathfrak{V}(i)}{\prod_{i=1}^{\leng}\prod_{j=1}^{d_{i}}\mathfrak{V}(j)}.\label{eq: Geassmannians vol}
\end{align}

On the remaining spaces, we define measures using the standard procedure
of (i) presenting a space as a quotient of a homogeneous manifold
by the action of a discrete group; (ii), identifying a fundamental
domain in the manifold for that action; (iii) defining the measure
on the space as the invariant measure on the manifold, restricted
to the fundamental domain (or rather, the pullback of this measure
through the inverse of the quotient map, which is one to one on the
fundamental domain). Recall the following construction of fundamental
domains representing 
\[
\begin{array}{ccccccccc}
\latspace n & = & \sl n(\RR)/\sl n(\ZZ) &  & \text{and} &  & \shapespace n & = & \so n(\RR)\backslash\sl n(\RR)/\sl n(\ZZ)\\
 &  &  &  &  &  &  & \simeq & P_{n}/\sl n(\ZZ)
\end{array}.
\]
Let $\symfund n\subset P_{n}$ be the standard Siegel fundamental
domain for right action of $\sl n(\ZZ)$ (see figure \ref{fig: fund dom SL(2,Z)}
for the case $n=2$).\footnote{The construction is essentially due to Siegel and is explicated in
\cite{Grenier_93,Schmidt_98} and \cite[VII]{HK_WellRoundedness}.} Let $\groupfund n$ be the lift of $\symfund n$ to $\sl n(\RR)$,
which is 
\begin{equation}
\groupfund n=\bigcup_{z\in\symfund n}K_{z}\cdot z\label{eq: fund SL(Z)}
\end{equation}
(\cite[Thm.~7.10 and Prop.~7.13]{HK_WellRoundedness}), where for
every $z\in\symfund n$, the notation $K_{z}\subset\so n\left(\RR\right)$
stands for a fundamental domain for the finite group of elements in
$\so{}(V_{\lat})$ that preserves $\lat$, $\sym^{+}\brac{\lat_{z}}$.
It is known (\cite[Lem.~6]{Schmidt_98}) that for almost every $z\in\interior{\symfund n}$
one has that $\sym^{+}\brac{\lat_{z}}=Z(K)$ where $Z\brac K$ is
the center of $K$. Letting $K_{\text{gen}}$ denote the generic fiber,
we have
\begin{equation}
\groupfund n=K_{\text{gen}}\cdot\interior{\symfund n}\cup\bigcup_{z\in\del\symfund n}K_{z}\cdot z\label{eq: fund SL(Z)-1}
\end{equation}
and therefore
\[
\mu_{G_{n}}\brac{\groupfund n}=\mu_{K_{n}}\brac{K_{n}/Z(K_{n})}\cdot\mu_{P^{\dprime}}\brac{\symfund n},=\mass{\mu_{K_{n}}}\mu_{P^{\dprime}}\brac{\symfund n}/\sgn n
\]
where
\[
\sgn n=\sbrac{K_{n}:Z(K_{n})}=\begin{cases}
1 & n\equiv1(\text{mod }2)\\
2 & n\equiv0(\text{mod }2)
\end{cases}.
\]
We refer to \cite{Volumes} for the fact that 
\begin{equation}
\mu_{G_{n}}\brac{\groupfund n}=\prod_{i=2}^{n}\zeta\left(i\right).\label{eq: Garrett}
\end{equation}
Recalling (\ref{eq: Haar K}) and the fact that $\leb{\sphere{i-1}}=i\cdot\mathfrak{V}(i)$
we have that
\[
\mass{\mu_{K_{n}}}=\prod_{i=1}^{n-1}\leb{\sphere i}=\prod_{i=2}^{n}\leb{\sphere{i-1}}=\frac{1}{2}\prod_{i=1}^{n}\leb{\sphere{i-1}}=\frac{1}{2}\prod_{i=1}^{n}i\cdot\mathfrak{V}(i);
\]
then, by (\ref{eq: fund SL(Z)-1}),
\[
\mu_{P_{n}}\brac{\symfund n}=\mu_{G_{n}}\brac{\groupfund n}\sbrac{K_{n}:Z(K_{n})}/\mass{\mu_{K_{n}}}=2\sgn n\frac{\prod_{i=2}^{d}\zeta\left(i\right)}{\prod_{i=1}^{n}i\cdot\mathfrak{V}(i)}\,.
\]
The sets of representatives $\symfund n$, $\groupfund n$ and $K^{\prime}$
allow us to construct fundamental domains for $\shapespace{\dvec}$
and $\flagspace{\dvec}$. Quite naturally, a fundamental domain for
$\shapespace{\underline{d}}\diffeo P^{\dprime}/G^{\dprime}\brac{\ZZ}$
is 
\[
\prod_{j=1}^{\leng}\symfund{d_{j}}\subset P^{\dprime},\;\text{with}\;\,\ensuremath{\mu_{P^{\dprime}}}(\prod_{j=1}^{\leng}\symfund{d_{j}})=\prod_{j=1}^{\leng}\ensuremath{\mu_{P_{d_{j}}}}\brac{\symfund{d_{j}}}.
\]
Also, 
\[
K^{\prime}\prod_{j=1}^{\leng}\groupfund{d_{j}}\subset K^{\prime}G^{\dprime}=\wc
\]
is a fundamental domain for 
\[
\flagspace{\dvec}=\sl n\left(\RR\right)/G^{\dprime}\left(\ZZ\right)N^{\prime}A^{\prime}=\wc/G^{\dprime}\left(\ZZ\right),
\]
and by (\ref{eq: measure Q}) its measure in $\wc$ is
\[
\mu_{\wc}\brac{K^{\prime}\prod_{j=1}^{\leng}\groupfund{d_{j}}}=\mu_{K^{\prime}}(K^{\prime})\mu_{G^{\dprime}}(\prod_{j=1}^{\leng}\groupfund{d_{j}}).
\]
Now, the measures on $\shapespace{\dvec}$ and $\flagspace{\dvec}$
are defined so that they correspond to the homogeneous measures on
the ambient manifolds, restricted to the fundemental domains. We list
them for future reference: 
\begin{defn}
\label{prop: spread models that we need}We let $\vol_{\latspace n}=\mu_{G_{n}}|_{\groupfund n}$
and $\vol_{\shapespace n}=\mu_{P_{n}}|_{\symfund n}$ , so that
\[
\mass{\vol_{\latspace n}}=\prod_{i=2}^{n}\zeta\left(i\right),\quad\mass{\vol_{\shapespace n}}=2\sgn n\frac{\prod_{i=2}^{d}\zeta\left(i\right)}{\prod_{i=1}^{n}i\cdot\mathfrak{V}(i)}.
\]
Set  $\vol_{\shapespace{\underline{d}}}=\prod_{j=1}^{\leng}\vol_{\shapespace{d_{j}}}$.
Let $\vol_{\flagspace{\dvec}}=\mu_{\wc}|_{K^{\prime}\prod_{j=1}^{\leng}\groupfund{d_{j}}}$
so that in particular
\begin{align*}
\mass{\vol_{\flagspace{\dvec}}} & =\mass{\vol_{\gras{\dvec,n}}}\cdot\prod_{j=1}^{\leng}\vol_{\latspace{d_{j}}}\\
 & =2^{\leng-1}\frac{n!}{d_{1}!\cdots d_{\leng}!}\frac{\prod_{i=1}^{n}\mathfrak{V}(i)}{\prod_{i=1}^{\leng}\prod_{j=1}^{d_{i}}\mathfrak{V}(j)}\cdot\prod_{j=1}^{\leng}\prod_{i=2}^{d_{j}}\zeta\left(i\right).
\end{align*}
The probability measures corresponding to $\vol_{\flagspace{\dvec}}$,
$\vol_{\shapespace{\dvec}}$, $\vol_{\gras{\dvec,n}}$ (and appearing
in Theorems \ref{thm: Counitng with supremum} and \ref{thm: counting anti canonical})
are denoted
\[
\vol_{\flagspace{\dvec}}^{1}\,,\,\vol_{\shapespace{\dvec}}^{1}\,,\,\vol_{\gras{\dvec,n}}^{1}.
\]
\end{defn}

The  fundamental domains in Def.\ \ref{prop: spread models that we need}
represent the corresponding spaces not only in terms of the measure.
They also have the property that the image of a ``nice enough''
set in the space, is a ``nice enough'' set in the associated fundamental
domain. We denote by $\parby{\wc}{\pairset}$ the image of $\pairset$
in $K^{\prime}\prod_{j=1}^{\leng}\groupfund{d_{j}}\subset\wc$, by
$P_{\symset}^{\dprime}$ the image of $\symset$ in $\prod\symfund{d_{j}}\subset P^{\dprime}$,
by $K_{\sphereset}^{\prime}$ the image of $\sphereset\subseteq\gras{\dvec,n}$
in $K^{\prime}$, and so forth. To make precise what we mean by ``nice
enough'', consider the following definition. 
\begin{defn}
\label{def: BCS}A subset $B$ of an orbifold $\manifold$ will be
called \emph{boundary controllable } if for every $x\in\mathcal{M}$
there is an open neighborhood $U_{x}$ of $x$ such that $U_{x}\cap\del B$
is contained in a finite union of embedded $C^{1}$ submanifolds of
$\mathcal{M}$, whose dimension is strictly smaller than $\dim\manifold$.
In particular, $B$ is boundary controllable if its (topological)
boundary consists of finitely many  subsets of embedded $C^{1}$
submanifolds.
\end{defn}

\begin{lem}
\label{prop: spread model BCS}If a subset of any of the spaces appearing
in Def.\ \ref{prop: spread models that we need} is boundary controllable,
then so is its image in the associated set of representatives (e.g.\
if $\pairset\subseteq\flagspace{\dvec}$ is boundary controllable,
then so is $\parby{\wc}{\pairset}$).
\end{lem}

\begin{proof}
If $\pairset\subseteq\flagspace{\dvec}$ is boundary controllable,
then so is its lift to $\wc$ (that is, its inverse image under the
quotient map $\wc\to\wc/G^{\dprime}\left(\ZZ\right)=\flagspace{\dvec}$).
The image $\parby{\wc}{\pairset}$ is the intersection of this lift
with the set of representatives $K^{\prime}\prod_{j=1}^{\leng}\groupfund{d_{j}}$,
which is also boundary controllable. The intersection of two boundary
controllable sets is boundary controllable. The other cases are handled
similarly.
\end{proof}
Lemma \ref{prop: spread model BCS} handles the connection between
boundary controllable sets in the spaces $\shapespace{\underline{d}}$,
$\latspace{\dvec}$ and $\flagspace{\dvec}$ \textemdash{} namely
spaces that are expressed as quotients by actions of discrete subgroups
\textemdash{} to boundary controllable sets in the associated sets
of representatives. But something similar can also be said for the
space $\gras{\dvec,n}$, namely that $K^{\prime}$ can be chosen to
satisfy the following property (\cite[Lem.~3.4 (ii)]{HK_gcd}):
\begin{condition}
\label{cond: K'}The set of representatives $K^{\prime}\subset K$
satisfies that if $\sphereset\subseteq\gras{\dvec,n}$ and $\Gset\subseteq K^{\dprime}$
are boundary controllable, then so is $K_{\sphereset}^{\prime}\Gset\subseteq K$.
\end{condition}

With the choice of volumes declared in Def.\ \ref{prop: spread models that we need},
we have the following relations between the spaces $\flagspace{\dvec}$,
$\shapespace{\underline{d}}$ and $\gras{\dvec,n}$.
\begin{prop}
\label{lem: lift to G''}\newcommandx\proj[2][usedefault, addprefix=\global, 1=, 2=]{\pi_{#1\to#2}}%
The following hold:
\begin{enumerate}
\item There exist natural projections from $\flagspace{\dvec}$ to $\shapespace{\underline{d}}$,
to $\gras{\dvec,n}$, and to $\shapespace{\underline{d}}\times\gras{\dvec,n}$.
\item Assume that $\pairset\subseteq\flagspace{\dvec}$ is the inverse image
of $\symset\times\sphereset\subseteq\shapespace{\underline{d}}\times\gras{\dvec,n}$
under the projection from part 1, $\proj[\flagspace{\dvec}][\shapespace{\underline{d}}\times\gras{\dvec,n}]$.
If $\symset\subseteq\shapespace{\dvec}$ and $\sphereset\subseteq\gras{\dvec,n}$
are measurable, then so is $\pairset$ and 
\[
\vol_{\flagspace{\dvec}}\brac{\pairset}=\vol_{\shapespace{\dvec}}\brac{\symset}\vol_{\gras{\dvec,n}}\brac{\sphereset}\cdot\prod_{i=1}^{\leng}\sgn{d_{i}}.
\]
\item If $\symset$ and $\sphereset$ are boundary controllable, then so
is $\pairset$. 
\end{enumerate}
\end{prop}

\begin{proof}
For the first part: the projection $\pi_{\flagspace{\dvec}\to\shapespace{\underline{d}}}$
is given by quotienting from the left by $\so n(\RR)$, the projection
$\pi_{\flagspace{\dvec}\to\gras{\dvec,n}}$ is the one induced by
the projection $\sl n(\RR)\to K/K^{\dprime}$ given by $kan\mapsto kK^{\dprime}$,
and the projection $\proj[\flagspace{\dvec}][\shapespace{\underline{d}}\times\gras{\dvec,n}]$
is the product of the latter two. For the second and third parts,
let $\groupset\subseteq\latspace{\dvec}$ be the inverse image of
$\symset$ under the natural projection $\latspace{\dvec}\to\shapespace{\underline{d}}$.
By (\ref{eq: fund SL(Z)-1}),
\[
\parby{\wc}{\pairset}=K_{\sphereset}^{\prime}G_{\groupset}^{\dprime}=\left(K_{\sphereset}^{\prime}K_{\text{gen}}^{\dprime}\cdot\brac{P_{\symset}^{\dprime}\cap\prod_{j=1}^{\leng}\interior{\symfund{d_{j}}}}\right)\cup\bigcup_{z\in P_{\symset}^{\dprime}\cap\del\prod_{j=1}^{\leng}\symfund{d_{j}}}K_{\sphereset}^{\prime}K_{z}^{\dprime}\cdot z.
\]
It is sufficient to show that the image of $\parby{\wc}{\pairset}$
under the diffeomorphism $\wc\to K\times P^{\dprime}$, which is obvious
from the above equation, is boundary controllable. As for the $K$
part of this image, we have that every $K_{z}^{\dprime}$ (and in
particular $K_{\text{gen}}^{\dprime}$) is boundary controllable in
$K^{\dprime}$, and hence by Condition \ref{cond: K'} every $K_{\sphereset}^{\prime}K_{z}^{\dprime}$
is boundary controllable in $K$. As for the $P^{\dprime}$ part of
this image, we have that: $P_{\symset}^{\dprime}$ is boundary controllable,
because of the assumption on $\symset$ and Lemma \ref{prop: spread model BCS};
$\interior{\symfund n}$ is boundary controllable, since $\symfund n$
is; hence $P_{\symset}^{\dprime}\cap\interior{\symfund n}$ is boundary
controllable, as an intersection of such. Moreover $P_{\symset}^{\dprime}\cap\del\symfund n$
is boundary controllable, since it is its own boundary. We conclude
that $\parby{\wc}{\pairset}$ is boundary controllable, and therefore
 $\pairset$ is. This proves the third part, but also that 
\[
\mu_{\wc}(\parby{\wc}{\pairset})=\mu_{K^{\prime}}(K_{\sphereset}^{\prime})\mu_{K^{\dprime}}(K_{\text{gen}}^{\dprime})\mu_{P^{\dprime}}(P_{\symset}^{\dprime})=\mu_{K^{\prime}}(K_{\sphereset}^{\prime})\mu_{P^{\dprime}}(P_{\symset}^{\dprime})/[K^{\dprime}:Z(K^{\dprime})],
\]
where in the first equality we used $\mu_{\wc}=\mu_{K^{\prime}}\times\mu_{G^{\dprime}}$
and $\mu_{G^{\dprime}}=\mu_{K^{\dprime}}\times\mu_{P^{\dprime}}$
((\ref{eq: Haar measure on G})). Since $[K^{\dprime}:Z(K^{\dprime})]=\prod_{i=1}^{\leng}\sgn{d_{i}}$,
and by Def.\ \ref{prop: spread models that we need}, we confirm
the third statement.
\end{proof}

\section{\label{sec: General Thm}A more general theorem}

The natural map from the space $\flagspace{\dvec}$ to $\shapespace{\dvec}\times\gras{\dvec,n}$
(Prop.\ \ref{lem: lift to G''}) hints at the fact that counting
primitive integral flags w.r.t\ their projections to $\flagspace{\dvec}$
would imply counting these flags w.r.t.\ their projections to $\shapespace{\dvec}\times\gras{\dvec,n}$,
which is the content of Theorems \ref{thm: Counitng with supremum}
and \ref{thm: counting anti canonical}. Indeed, in this section we
state a counting theorem with $\flagspace{\dvec}$, from which we
deduce Theorems \ref{thm: Counitng with supremum} and \ref{thm: counting anti canonical}.
Recall from Section \ref{sec: Flags} that the projection of a flag
of lattices $\flag\brac{\ZZ}$ to the space $\flagspace{\dvec}$ is
denoted by $\unilat{\flag\brac{\ZZ}}$. In what follows, we say that
$\pairset\subseteq\flagspace{\dvec}$ is bounded if $\proj[\flagspace{\dvec}][\shapespace{\underline{d}}]\brac{\pairset}$
is.

\begin{thm}
\label{thm: general thm}Let $\errexp_{n}=(4n^{2}/\left\lceil \left(n-1\right)/2\right\rceil )^{-1}$,
and assume that $\pairset\subseteq\flagspace{\dvec}$ is boundary
controllable. The number of positive primitive integral $\dvec$-flags
$\flag\brac{\ZZ}$ with $\unilat{\flag\brac{\ZZ}}\in\pairset$ and
$H\brac{\flag\brac{\ZZ}}\leq e^{T}$ is 
\[
2^{\leng-1}c_{\dvec,n}\cdot\vol_{\flagspace{\dvec}}^{\,1}\brac{\pairset}\cdot e^{\ind T}+\text{error term}
\]
where 
\begin{equation}
\ind=\begin{cases}
2n-d_{1}-d_{\leng} & \text{if \ensuremath{H=H_{\infty}}}\\
1 & \text{if \ensuremath{H=H_{\ac}}}
\end{cases}\label{eq: choose height}
\end{equation}
and the error term is $O_{\e}\brac{e^{\ind T\brac{1-\frac{\errexp_{n}n}{\b}+\e}}}$
for every $\e>0$, where 
\[
\b=\begin{cases}
1 & \text{if \ensuremath{\pairset} is bounded}\\
2\left(n-\leng\right)\left(n^{2}-1\right)+n^{2} & \text{if \ensuremath{\pairset} is not bounded}
\end{cases}.
\]
 
\end{thm}

\begin{proof}[Proof of Theorems \ref{thm: Counitng with supremum} and \ref{thm: counting anti canonical}
assuming Theorem \ref{thm: general thm}.]
Let $\symset\subseteq\shapespace{\dvec}$ and $\sphereset\subseteq\grass{\dvec,n}$
be boundary controllable. We denote by $\tilde{\sphereset}\subseteq\gras{\dvec,n}$
the lift of $\sphereset$ to $\gras{\dvec,n}$, which is also boundary
controllable. By Proposition \ref{lem: lift to G''} the set 
\[
\pairset=\proj[\flagspace{\dvec}][\shapespace{\underline{d}}\times\gras{\dvec,n}]^{-1}(\symset\times\tilde{\sphereset})
\]
is boundary controllable and $\vol_{\flagspace{\dvec}}\brac{\pairset}=\vol_{\shapespace{\dvec}}\brac{\symset}\vol_{\gras{\dvec,n}}\brac{\tilde{\sphereset}}\prod_{i=1}^{\leng}\sgn{d_{i}}$.
In particular, for $\symset=\shapespace{\dvec}$ and $\tilde{\sphereset}=\gras{\dvec,n}$
we have that
\[
\mass{\vol_{\flagspace{\dvec}}}=\mass{\vol_{\shapespace{\dvec}}}\mass{\vol_{\gras{\dvec,n}}}\prod_{i=1}^{\leng}\sgn{d_{i}}.
\]
Then, by Theorem \ref{thm: general thm}, the number of positive primitive
integral $\dvec$-flags $\flag\brac{\ZZ}$ with $\unilat{\flag\brac{\ZZ}}\in\pairset$
and $H\brac{\flag\brac{\ZZ}}\leq e^{T}$ is 
\begin{align*}
 & \approx2^{\leng-1}c_{\dvec,n}\cdot\frac{\vol_{\flagspace{\dvec}}\brac{\pairset}}{\mass{\vol_{\flagspace{\dvec}}}}\cdot e^{\ind T}=2^{\leng-1}c_{\dvec,n}\cdot\frac{\vol_{\shapespace{\dvec}}\brac{\symset}\,\vol_{\gras{\dvec,n}}\brac{\tilde{\sphereset}}}{\mass{\vol_{\shapespace{\dvec}}}\mass{\vol_{\gras{\dvec,n}}}}\cdot e^{\ind T}\\
 & =2^{\leng-1}c_{\dvec,n}\cdot\vol_{\shapespace{\dvec}}^{1}\brac{\symset}\vol_{\gras{\dvec,n}}^{1}\brac{\tilde{\sphereset}}\cdot e^{\ind T}=c_{\dvec,n}\cdot\vol_{\shapespace{\dvec}}^{1}\brac{\symset}\vol_{\grass{\dvec,n}}^{1}\brac{\sphereset}\cdot e^{\ind T},
\end{align*}
where in the last equality we used (\ref{eq: Geassmannians vol}).
\end{proof}
The rest of this article is devoted to proving Theorem \ref{thm: general thm}.

\section{\label{sec: Integral matrices representing primitive vectors}Integral
matrices correspond to integral flags}

The goal of this section is to translate the statement of Theorem
\ref{thm: general thm} into a counting problem of integral matrices.
To do so, we establish a correspondence between integral unimodular
$\dvec$-flags and integral matrices in a fundamental domain of the
following discrete group of $\sl n\brac{\RR}$: 
\[
\disgrp^{\prime}=\left(N^{\prime}\rtimes G^{\dprime}\right)\left(\ZZ\right)=\left[\begin{smallmatrix}\sl{d_{1}}\left(\ZZ\right) & \ZZ^{d_{1}\times d_{2}} & \cdots & \ZZ^{d_{1}\times d_{\leng}}\\
0 & \ddots & \ddots & \vdots\\
\vdots & \ddots & \sl{d_{\leng-1}}\left(\ZZ\right) & \ZZ^{d_{\leng-1}\times d_{\leng}}\\
0 & \cdots & 0 & \sl{d_{\leng}}\left(\ZZ\right)
\end{smallmatrix}\right].
\]

\begin{prop}
\label{prop: prim Z flags correspond to integral matrices}There
exists a bijection $\flag\leftrightarrow\ga_{\flag}$ between integral
unimodular $\dvec$-flags and integral matrices in a fundamental domain
of $\sl n\left(\RR\right)\curvearrowleft\disgrp^{\prime}$, that sends
a unimodular $\dvec$-flag $\flag$ to $\ga_{\flag}$, the unique
integral matrix in the fundamental domain whose columns span $\flag$. 
\end{prop}

\begin{proof}
The direction $\Leftarrow$ is simple: given $\ga\in\funddom\cap\sl n\left(\ZZ\right)$,
its columns span $\ZZ^{n}$ hence by definition a $\dvec$ partition
of its columns spans a unimodular integral $\dvec$\textendash flag.
In the opposite direction, let $\base_{1}|\cdots|\base_{\leng}$ be
a basis for $\flag$. Since $\flag$ is primitive, we may assume that
$\base_{\leng}$ is also a basis for $\ZZ^{n}$; let $\ga\in\sl n\left(\ZZ\right)$
be a matrix having this basis in its columns. The orbit $\ga\cdot\disgrp$
meets $\funddom$ in a single point, $\ga_{\flag}$. 
\end{proof}
Let us construct an explicit fundamental domain for $\disgrp^{\prime}$.
 Denote 
\[
\cube=\text{the unit cube \ensuremath{\left(-1/2,1/2\right]^{\dim(N^{\prime})}}},
\]
let $\parby{N^{\prime}}{\cube}$ be the image of $\cube$ in $N^{\prime}$,
and set
\[
\funddom=K^{\prime}\parby{G^{\dprime}}{\groupfund{d_{1}}\times\cdots\times\groupfund{d_{\leng}}}A^{\prime}\parby{N^{\prime}}{\cube}\,.
\]
It is easy to see  that $\funddom$ is a fundamental domain for the
right action of $\disgrp^{\prime}$ on $\sl n\left(\RR\right)$. 

Proposition \ref{prop: prim Z flags correspond to integral matrices}
confirms that every unimodular integral flag is represented by a unique
integral matrix in $\funddom$; the next step is to verify how the
different properties of the flag \textemdash{} e.g.\ its shape, or
the covolumes of the lattices that compose it \textemdash{} are exhibited
in the Refined Iwasawa components of the matrix that spans the flag. 

For $g\in\gl n(\RR)$, recall the notation for the primitive $\dvec$-flag
of lattices spanned by the columns of $g$:
\[
\flag_{g}\brac{\ZZ}=\:(\,\Flat{\lat_{g}}1<\Flat{\lat_{g}}2<\cdots<\Flat{\lat_{g}}{\leng}=\lat_{g}<\RR^{n}\,).
\]
By setting 
\[
D_{j}=\sum_{i=1}^{j}d_{i},
\]
we have that each $\Flat{\lat_{g}}j$ is the $\ZZ$-span of the first
$D_{j}$ columns of $g$. For $D_{j-1}+1\leq i\leq D_{j}$, we let
\[
\brac{\Flat{\lat_{g}}j}^{i}
\]
be the subgroup of $\Flat{\lat_{g}}j$ spanned by the columns $\sbrac{D_{j-1}+1,i}$
of $g$. The consecutive quotients are
\[
\qlat_{j}(g)=\Flat{\lat_{g}}j/\Flat{\lat_{g}}{j-1},
\]
that is, each $\qlat_{j}(g)$ is spanned by the cosets of $\Flat{\lat_{g}}{j-1}$
that are generated by the columns $\sbrac{D_{j-1}+1,D_{j}}$ of $g$.
Let 
\[
\qlat_{j}\brac g^{i}
\]
be the subgroup of $\qlat_{j}\brac g$ spanned by the cosets corresponding
to columns $\sbrac{D_{j-1}+1,i}$.
\begin{prop}
\label{prop: explicit RI coordinates of g}Assume $g\in\sl n\left(\RR\right)$
is written in Refined Iwasawa coordinates as 
\[
g=kan=qa^{\prime}n^{\prime}
\]
where \setlength{\abovedisplayskip}{0pt} 
\setlength{\belowdisplayskip}{5pt}
\begin{align*}
a^{\prime} & =a_{t_{1},\ldots,t_{\leng-1}}^{\prime}\\
q & =k^{\prime}g^{\dprime}=k^{\prime}k^{\dprime}p^{\dprime}a_{\Flat{\underline{s}}1,\ldots,\Flat{\underline{s}}{\leng}}^{\dprime}n^{\dprime}.
\end{align*}
for every $j=1,\ldots,\leng$, let $g^{\dprime}=\diag{g_{d_{1}},\ldots,g_{d_{\leng}}}$,
$g_{d_{j}}^{\dprime}=\diag{I_{d_{1}},\ldots,g_{d_{j}},\ldots,\idmat{d_{\leng}}}$,
and similarly for $p^{\dprime}$. The Refined Iwasawa components of
$g$ represent parameters related to $\flag_{g}\brac{\ZZ}$ in the
following way:
\[
\begin{array}{ccccc}
(i) & \flag_{k^{\prime}} & = & \flag_{g}\\
(ii) & e^{t_{j}} & = & \covol{\Flat{\lat_{g}}j}\\
(ii)^{\prime} & e^{t_{j}-t_{j-1}} & = & \covol{\qlat_{j}\brac g} & (t_{0}=t_{\leng}=0)\\
(iii) & e^{\frac{i\brac{t_{j}-t_{j-1}}}{d_{j}}-\frac{\Flat{s_{i}}j}{2}} & = & \covol{\qlat_{j}\brac g^{i}}\\
(iii)^{\prime} & e^{\frac{\hat{i}t_{j}+(d_{j}-\hat{i})t_{j-1}}{d_{j}}-\frac{\Flat{s_{\hat{i}}}j}{2}} & = & \covol{\Flat{\lat}j{}^{i}} & \left(\substack{D_{j-1}+1\leq i\leq D_{j}\\
\hat{i}=i-d_{j-1}
}
\right)\\
(iv) & \unilat{\flag_{q}\brac{\ZZ}} & = & \unilat{\flag_{g}(\ZZ)}\\
(v) & \ensuremath{\shape{\qlat_{j}}} & = & \shape{\lat_{p_{d_{j}}}}
\end{array}
\]
\end{prop}

\begin{proof}
Since the columns of $k$ are obtained by performing the Gram-Schmidt
orthogonalization procedure on the columns of $g$, we have for every
$1\leq d\leq n$ that the first $d$ columns of $k$ span the same
space as the first $d$ columns of $g$. Hence $\flag_{g}=\flag_{k}$.
Since $k^{\prime}=k\left(k^{\dprime}\right)^{-1}$, where $k^{\dprime}\in G^{\dprime}$
fixes $\flag_{g}$, we deduce that $(i)$ indeed holds.  

Now write $g\brac{n^{\prime}}^{-1}\brac{a^{\prime}}^{-1}=q$; by definition,
right multiplication by an element of $N^{\prime}A^{\prime}$ does
not change the projection to $\flagspace{\dvec}$. This observation
immediately proves part $(iv)$. As the shape of each $g_{d_{j}}^{\dprime}$
is the same as that of $p_{d_{j}}^{\dprime}$, part $(v)$ follows
from part $(iv)$. 

Notice that since $g^{\dprime}$ is in $G^{\dprime}$, the lattice
flag $\flag_{g^{\dprime}}\brac{\ZZ}$ is unimodular. Furthermore,
as $q$ is a rotation of $g^{\dprime}$, the same holds for $\flag_{q}\brac{\ZZ}$.
 As a result, using the fact that right multiplication by $n^{\prime}$
does not change the flag $\flag_{q}\brac{\ZZ}$, we have\setlength{\abovedisplayskip}{0pt} 
\setlength{\belowdisplayskip}{6pt}
\[
\covol{\qlat_{j}^{i}(g)}=\covol{\qlat_{j}^{i}(qa^{\prime})}=\covol{\qlat_{j}^{i}(a^{\prime})}=\covol{(a^{\prime})^{D_{j}+i}}=e^{\frac{i\left(t_{j}-t_{j-1}\right)}{d_{j}}-\frac{\Flat{s_{i}}j}{2}}.
\]
This proves parts $(iii)$ and $(ii)^{\prime}$.

Notice that if $\lat_{1}<\lat_{2}$ are lattices, then
\[
\covol{\lat_{2}}=\covol{\lat_{1}}\covol{\lat_{2}/\lat_{1}}.
\]
As a result, parts $(ii)$ and $(iii)^{\prime}$ follow from parts
$(iii)$ and $(ii)^{\prime}$.
\end{proof}
\begin{rem}
Proposition \ref{prop: explicit RI coordinates of g} above explicates
the projections from $\flagspace{\dvec}$ to $\shapespace{\dvec}$
and to $\gras{\dvec,n}$ that appear in Proposition \ref{lem: lift to G''}:
the projection of $\flag_{g}\brac{\ZZ}$ to $\flagspace{\dvec}$ is
$q=k^{\prime}k^{\dprime}p^{\dprime}$, while $p^{\dprime}$ represents
the projection to $\shapespace{\dvec}$, and $k^{\prime}$ represents
the projection  to $\gras{\dvec,n}$.
\end{rem}

We can now define subsets of $\funddom$ that capture only the integral
matrices corresponding to unimodular lattice flags with certain shape,
direction, and height. For $T>0$ denote: 
\begin{align*}
\parby{A^{\prime}}T & :=\{a_{t_{1},\ldots,t_{\leng-1}}^{\prime}:\log\left(H_{\infty}\brac{t_{1},\ldots,t_{\leng-1}}\right)\leq T\}\\
 & =\{a_{t_{1},\ldots,t_{\leng-1}}^{\prime}:0\leq t_{i}\leq T\;\forall i=1,\ldots,\leng-1
\end{align*}
and
\begin{align*}
\bypar{A^{\prime}}T & :=\{a_{t_{1},\ldots,t_{\leng-1}}^{\prime}:\log\left(H_{\ac}\brac{t_{1},\ldots,t_{\leng-1}}\right)\leq T\}\\
 & =\{a_{t_{1},\ldots,t_{\leng-1}}^{\prime}:\sum_{i=1}^{\leng-1}t_{i}\brac{d_{i}+d_{i+1}}\leq T\}.
\end{align*}
Also, let
\[
\byparby{A^{\prime}}T=\begin{cases}
\parby{A^{\prime}}T & \text{if \ensuremath{H=H_{\infty}}}\\
\bypar{A^{\prime}}T & \text{if \ensuremath{H=H_{\ac}}}
\end{cases}.
\]

\begin{notation}
\label{nota: sets we count in}For $T>0$ and $\pairset\subseteq\flagspace{\dvec,n}$,
consider 
\[
\funddom_{T}\brac{\pairset}=\,\funddom\,\cap\,\left\{ g=qa^{\prime}n^{\prime}:q\in\parby{\wc}{\pairset},a^{\prime}\in\parby{A^{\prime}}T\right\} =\parby{\wc}{\pairset}\parby{A^{\prime}}T\parby{N^{\prime}}{\cube}.
\]
Similarly, denote 
\[
\bypar{\funddom}T\brac{\pairset}=\,\funddom\,\cap\,\left\{ g=qa^{\prime}n^{\prime}:q\in\parby{\wc}{\pairset},a^{\prime}\in\bypar{A^{\prime}}T\right\} =\parby{\wc}{\pairset}\:\bypar{A^{\prime}}T\parby{N^{\prime}}{\cube}
\]
for the analogous set where $\parby{A^{\prime}}T$ is replaced by
$\bypar{A^{\prime}}T$. Finally,
\[
\byparby{\funddom}T\brac{\pairset}=\parby{\wc}{\pairset}\:\byparby{A^{\prime}}T\parby{N^{\prime}}{\cube}=\begin{cases}
\funddom_{T}\brac{\pairset} & \text{if \ensuremath{H=H_{\infty}}}\\
\bypar{\funddom}T\brac{\pairset} & \text{if \ensuremath{H=H_{\ac}}}
\end{cases}.
\]
\end{notation}

The following is immediate from Propositions \ref{prop: prim Z flags correspond to integral matrices}
and \ref{prop: explicit RI coordinates of g}: 
\begin{cor}
\label{cor: the sets we should count in}Consider the correspondence
$\flag\leftrightarrow\ga_{\flag}$ between integral umimodular $\dvec$\textendash flags
and matrices in $\funddom\cap\sl n\left(\ZZ\right)$, and let $T>0$.
Then 
\begin{eqnarray*}
\funddom_{T}\brac{\pairset}\cap\sl n\left(\ZZ\right) & = & \left\{ \ga_{\flag}:H_{\infty}\brac{\flag}\leq e^{T},\unilat{\flag\brac{\ZZ}}\in\pairset\right\} ,\\
\bypar{\funddom}T\brac{\pairset}\cap\sl n\left(\ZZ\right) & = & \left\{ \ga_{\flag}:H_{\ac}\brac{\flag}\leq e^{T},\unilat{\flag\brac{\ZZ}}\in\pairset\right\} .
\end{eqnarray*}
\end{cor}

\section{\label{sec: volumes}Some volume computations}

The goal of this section is to compute the volumes of the sets $\funddom_{T}\brac{\pairset}$
and $\bypar{\funddom}T\brac{\pairset}$, introduced in Notation \ref{nota: sets we count in}.
From now on, we will abbreviate and let $\mu=\mu_{G_{n}}$. 
\begin{prop}
\label{prop: volumes}
\[
\mu\brac{\funddom_{T}\brac{\pairset}}=\frac{\vol_{\flagspace{\dvec}}\brac{\pairset}}{\prod_{j=1}^{\leng-1}\brac{d_{j}+d_{j+1}}}\cdot e^{T\brac{2n-d_{1}-d_{\leng}}}+O(e^{T\left(2n-d_{1}-d_{\leng}-2\right)}).
\]
 and 
\[
\mu\brac{\bypar{\funddom}T\brac{\pairset}}=\frac{\vol_{\flagspace{\dvec}}\brac{\pairset}}{\prod_{j=1}^{\leng-1}\brac{d_{j}+d_{j+1}}}\cdot e^{T}\sum_{i=0}^{\leng-2}(-1)^{\leng-2-i}\cdot\frac{T^{i}}{i!}+O(1).
\]
\end{prop}

For the proof, consider the following computational lemma. 
\begin{lem}
\label{lem: Florian measure}Let 
\[
f_{m}(T)=\int_{0}^{T}e^{x_{m}}\int_{0}^{T-x_{m}}e^{x_{m-1}}\cdots\int_{0}^{T-x_{2}-\cdots-x_{m}}e^{x_{1}}dx_{1}\cdots dx_{m}.
\]
Then 
\[
f_{m}(T)=e^{T}\cdot\sum_{i=0}^{m-1}(-1)^{m-i-1}\frac{T^{i}}{i!}+(-1)^{m}.
\]
\end{lem}

\begin{proof}
Notice that for $m>1$
\[
f_{m}(T)=\int_{0}^{T}f_{m-1}(T-s_{m})e^{s_{m}}ds_{m},
\]
and so we will prove the claim by induction on $m$. When $m=1$ we
have that:
\[
f_{1}(T)=\int_{0}^{T}e^{x_{1}}dx_{1}=e^{T}-1=e^{T}\sum_{i=0}^{0}(-1)^{-i}\frac{T^{i}}{i!}+(-1)^{1}.
\]
Assume that the claim holds for some $m-1\geq1$. We have:
\begin{align*}
f_{m}(T) & =\int_{0}^{T}f_{m-1}(x_{m}-T)e^{x_{m}}dx_{m}\\
 & =\int_{0}^{T}\left(e^{T}\cdot\sum_{i=0}^{m-2}(-1)^{m-i-2}\frac{(T-x_{m})^{i}}{i!}+(-1)^{m-1}e^{x_{m}}\right)dx_{m}\\
 & =e^{T}\cdot\sum_{i=0}^{m-2}(-1)^{m-i-1}\frac{(T-x_{m})^{i+1}}{(i+1)!}+(-1)^{m-1}e^{x_{m}}|_{x_{m}=0}^{x_{m}=T}\\
 & =e^{T}\cdot\sum_{i=0}^{m-2}(-1)^{m-(i+1)-1}\frac{T^{i+1}}{(i+1)!}+(-1)^{m-1}e^{T}+(-1)^{m}.
\end{align*}
Set $j=i+1$ and then
\begin{align*}
 & =e^{T}\cdot\sum_{j=1}^{m-1}(-1)^{m-j-1}\frac{T^{j}}{j!}+(-1)^{m-1}e^{T}+(-1)^{m}\\
 & =e^{T}\cdot\sum_{j=0}^{m-1}(-1)^{m-j-1}\frac{T^{j}}{j!}+(-1)^{m}.\qedhere
\end{align*}
\end{proof}
\begin{proof}[Proof of Proposition \ref{prop: volumes}]
By definition, $\byparby{\funddom}T\brac{\pairset}=\parby{\wc}{\pairset}\:\byparby{A^{\prime}}T\parby{N^{\prime}}{\cube}$.
From (\ref{eq: Haar measure on G}), 
\begin{eqnarray*}
\mu\brac{\byparby{\funddom}T\brac{\pairset}} & = & \mu_{Q}\brac{\parby{\wc}{\pairset}}\mu_{A^{\prime}}\brac{\byparby{A^{\prime}}T}\mu_{N^{\prime}}\brac{\parby{N^{\prime}}{\cube}},
\end{eqnarray*}
where by Proposition \ref{prop: spread models that we need} and the
definition of $\mu_{N^{\prime}}$,
\begin{eqnarray*}
\mu_{Q}\brac{\parby{\wc}{\pairset}} & = & \vol_{\flagspace{\dvec}}\brac{\pairset}\\
\mu_{N^{\prime}}\brac{\parby{N^{\prime}}{\cube}} & = & \leb{\cube}=1\,.
\end{eqnarray*}
It is therefore left to compute the $\mu_{A^{\prime}}$-volumes of
$\parby{A^{\prime}}T$ and $\bypar{A^{\prime}}T$. Referring to (\ref{eq: A' measure}),
\begin{align*}
\mu_{A^{\prime}}\brac{\parby{A^{\prime}}T} & =\prod_{j=1}^{\leng-1}\int_{0}^{T}e^{\left(d_{j}+d_{j+1}\right)t_{j}}dt_{j}=\prod_{j=1}^{\leng-1}\frac{e^{\left(d_{j}+d_{j+1}\right)T}-1}{d_{j}+d_{j+1}}\\
 & =\frac{e^{T\cdot\sum_{j=1}^{\leng-1}\left(d_{j}+d_{j+1}\right)}}{\prod_{j=1}^{\leng-1}\brac{d_{j}+d_{j+1}}}+O(e^{T\cdot\brac{\sum_{j=1}^{\leng-1}\left(d_{j}+d_{j+1}\right)-\min_{j}\cbrac{d_{j}+d_{j+1}}}}).
\end{align*}
Note that 
\[
\sum_{j=1}^{\leng-1}\left(d_{j}+d_{j+1}\right)=\sum_{j=1}^{\leng-1}d_{j}+\sum_{j=2}^{\leng}d_{j}=\brac{n-d_{\leng}}+\brac{n-d_{1}}
\]
(since $d_{1}+\cdots+d_{\leng}=n$). Therefore 
\[
\mu_{A^{\prime}}(A_{T}^{\prime})=\prod_{j=1}^{\leng-1}\frac{1}{d_{j}+d_{j+1}}e^{T\cdot\left(2n-d_{1}-d_{\leng}\right)}+O(e^{T\cdot\left(2n-d_{1}-d_{\leng}-2\right)}).
\]
Moving on to $\bypar{A^{\prime}}T$ and applying Lemma \ref{lem: Florian measure},
\begin{align*}
\mu_{A^{\prime}}\brac{\bypar{A^{\prime}}T} & =\frac{f_{\leng-1}(T)}{\prod_{j=1}^{\leng-1}\brac{d_{j}+d_{j+1}}}f_{\leng-1}(T)=\\
 & =\frac{e^{T}}{\prod_{j=1}^{\leng-1}\brac{d_{j}+d_{j+1}}}\cdot\sum_{i=0}^{\leng-2}(-1)^{\leng-2-i}\frac{T^{i}}{i!}+O(1).\qedhere
\end{align*}
\end{proof}

\section{\label{sec: counting with GN}Counting lattice points}

As Corollary \ref{cor: the sets we should count in} suggests, Theorem
\ref{thm: general thm} is proved by counting $\sl n\left(\ZZ\right)$
matrices inside the sets $\funddom_{T}\brac{\pairset}$ and $\bypar{\funddom}T\brac{\pairset}$.
More precisely, this theorem follows from a counting lattice points
statement of the form
\[
\#\brac{\byparby{\funddom}T\brac{\pairset}\cap\sl n\left(\ZZ\right)}=\mu\brac{\sl n\brac{\RR}/\sl n(\ZZ)}^{-1}\cdot\mu\brac{\byparby{\funddom}T\brac{\pairset}}+O\brac{\mu\brac{\byparby{\funddom}T\brac{\pairset}}}^{\kappa}
\]
for some $0<\kappa<1$. However, the sets $\byparby{\funddom}T\brac{\pairset}$
are not compact when $\pairset$ is not compact (which is equivalent
to the fact that the projection of $\pairset$ to $\shapespace{\dvec}$
is not compact), even though they have finite volume and contain a
finite amount of lattice points. The way to address this problem is
by splitting each set $\byparby{\funddom}T\brac{\pairset}$ into a
compact subset and an ``infinite tail''. In the present section
we define a family of compact subsets of $\byparby{\funddom}T\brac{\pairset}$,
and apply a known ergodic method \cite{GN1} to count lattice points
in this family. In Section \ref{sec: Counting the cusp}, we apply
direct counting to bound the amount of points in the ``infinite tail''. 

The family of compact subsets of $\byparby{\funddom}T\brac{\pairset}$
that we consider in this section is defined as follows. For  
\begin{equation}
\Svec=\brac{\Flat{S_{1}}1,\ldots,\Flat{S_{d_{1}-1}}1,\dots,\Flat{S_{1}}{\leng},\ldots,\Flat{S_{d_{\leng}-1}}{\leng}}>\underline{0}\label{eq: Svec}
\end{equation}
and a subset $\Gset\subset\sl n\brac{\RR}$, let $\trunc{\Gset}{\Svec}$
denote the set 
\[
\trunc{\Gset}{\Svec}=\Gset\cap\cbrac{g=ka^{\prime}a_{\svec}^{\dprime}n:\Flat sj_{i}\leq\Flat{S_{i}}j\;\,\forall1\leq i<d_{j},1\leq j\leq\leng}.
\]
  Specifically, let 
\begin{align*}
\trunc{\byparby{\funddom}T}{\Svec} & =K^{\prime}\parby{G^{\dprime}}{\trunc{\groupfund{d_{1}}}{\Svec^{(1)}}\times\cdots\times\trunc{\groupfund{d_{\leng}}}{\Svec^{(\leng)}}}\,\byparby{A^{\prime}}T\parby{N^{\prime}}{\cube}.
\end{align*}
One can deduce from the proof of Proposition \ref{prop: volumes}
that
\begin{equation}
\begin{aligned}\mu\brac{\byparby{\funddom}T-\trunc{\byparby{\funddom}T}{\Svec}}=O\brac{e^{\ind T-S_{\min}}}\end{aligned}
\label{eq: volumes truncated}
\end{equation}
where $S_{\min}=\min_{i,j}S_{i}^{(j)}$ and $\ind$ was defined in
(\ref{eq: choose height}). Indeed, 
\begin{align*}
\mu\left(\byparby{\funddom}T-\trunc{\byparby{\funddom}T}{\Svec}\right) & =O\left(\mu_{A}\brac{\byparby{A^{\prime}}TA^{\dprime}-\byparby{A^{\prime}}T\trunc{\brac{A^{\dprime}}}{\underline{S}}}\right)
\end{align*}
where $\mu_{A}\brac{\byparby{A^{\prime}}TA^{\dprime}-\byparby{A^{\prime}}T\trunc{\brac{A^{\dprime}}}{\underline{S}}}$
equals $\mu_{A^{\prime}}(\byparby{A^{\prime}}T)\mu_{A^{\dprime}}\brac{A^{\dprime}-\trunc{\brac{A^{\dprime}}}{\underline{S}}}$,
as $\mu_{A}=\mu_{A^{\prime}}\times\mu_{A^{\dprime}}$. In the proof
of Prop.\ \ref{prop: volumes} it was shown that $\mu_{A^{\prime}}(\byparby{A^{\prime}}T)\ll e^{\ind T}$;
now
\begin{align*}
\mu_{A^{\dprime}}\left(A^{\dprime}-\trunc{\brac{A^{\dprime}}}{\underline{S}}\right) & =\int_{\left[0,\infty\right)^{n-\leng}-\prod_{j=1}^{n-\leng}\left[0,S_{j}\right]}d\mu_{A^{\dprime}}\\
 & \ll e^{\brac{2n-d_{1}-d_{\leng}}T}\cdot\sum_{j=1}^{n-\leng}\int_{\left(S_{j}\leq s_{j}<\infty\right)\times\prod_{i\neq j}\left(0\leq s_{i}<\infty\right)}\frac{du_{1}\cdots du_{n-\leng}}{e^{u_{1}+\cdots+u_{n-\leng}}}\\
 & \ll e^{\brac{2n-d_{1}-d_{\leng}}T}\cdot\left(e^{-s_{1}}+\cdots+e^{-s_{n-\leng}}\right)\\
 & \ll e^{\brac{2n-d_{1}-d_{\leng}}T}\cdot e^{-\min S_{i}}.
\end{align*}

In order to count $\sl n(\ZZ)$ elements in $\trunc{\funddom_{T}}{\Svec}$
and $\trunc{\bypar{\funddom}T}{\Svec}$, we will employ a method that
was developed by Gorodnik and Nevo in \cite{GN1}. This method produces
counting statements for lattice subgroups of Lie groups, including
an error estimate. The bound on the error exponent involves a parameter
$\errexp$ that we now turn to define. 

Let $\Lat$ be a lattice subgroup of a simple algebraic Lie group
$G$, that is, $\Lat$ is discrete and $\haar_{G}(G/\Lat)<\infty.$
There exists $p\in\mathbb{N}$ for which the matrix coefficients $\dbrac{\pi_{G/\Gamma}^{0}u,v}$
are in $L^{p+\e}\left(G\right)$ for every $\e>0$, with $u,v$ lying
in a dense subspace of $L_{0}^{2}\left(G/\Gamma\right)$ (see \cite[Thm.~5.6]{GN_book}).
Let $p\left(\Gamma\right)$ be the smallest among these $p$'s, and
denote 
\[
m\left(\Lat\right)=\begin{cases}
1 & \text{if \ensuremath{p=2},}\\
2\left\lceil p\left(\Gamma\right)/4\right\rceil  & \text{otherwise}
\end{cases},
\]
\begin{equation}
\errexp\left(\Lat\right)=\frac{1}{2m\left(\Gamma\right)\left(1+\dim G\right)}\in\left(0,1\right).\label{eq: tau}
\end{equation}
We say that a family $\cbrac{\Gset_{T}}_{T>0}$ of subsets of a Lie
group is \emph{Lipschitz well rounded} (\cite[Def.~1.1]{GN1}) if
there exist two constants $T_{0},C_{\Gset}>0$ that do not depend
on $T$ such that for every $T>T_{0}$ and $\e>0$
\[
\frac{\mu\brac{\Gset_{T}^{\left(+\e\right)}}}{\mu\brac{\Gset_{T}^{\left(-\e\right)}}}:=\frac{\mu\brac{\bigcup_{u,v\in\nbhd{\e}{}}\nbhd{\e}{}\Gset_{T}\nbhd{\e}{}}}{\mu\brac{\bigcap_{u,v\in\nbhd{\e}{}}\nbhd{\e}{}\Gset_{T}\nbhd{\e}{}}}\leq1+C_{\Gset}\e.
\]
 The goal of this section is to prove the following proposition. 
\begin{prop}
\label{thm: Counting with S and W}Let $\Lat<\sl n\left(\RR\right)$
be a lattice subgroup, $\errexp=\errexp\left(\Lat\right)$ as in (\ref{eq: tau}),
$\Svec$ as in (\ref{eq: Svec}),
\begin{align*}
\sumS & =\text{sum of components of \ensuremath{\Svec}},
\end{align*}
 and $\lm_{n}=\frac{n^{2}}{2\left(n^{2}-1\right)}$. If $\pairset\subseteq\flagspace{\dvec}$
is boundary controllable, then for every $0<\e<\errexp$ and $T\geq\frac{\sumS}{\errexp\lm_{n}\ind}+O_{\pairset}(1)$
\[
\#\brac{\trunc{\byparby{\funddom}T}{\Svec}\brac{\pairset}\cap\Lat}=\frac{\mu\brac{\trunc{\byparby{\funddom}T}{\Svec}\brac{\pairset}}}{\mu\left(\Lat\backslash G\right)}+O_{\pairset,\e}\left(e^{\sumS/\lm_{n}}\mu\brac{\trunc{\byparby{\funddom}T}{\Svec}\brac{\pairset}}^{\left(1-\errexp+\e\right)}\right).
\]
\end{prop}

The proof will make use of the following theorem:

\begin{thm}[{\cite[Thm.~1.9, Thm.~4.5, and Rem.~1.10]{GN1}}]
\label{thm: GN Counting thm}Let $G$ be an algebraic simple Lie
group, and $\Lat<G$ a lattice subgroup. Assume that $\{\Gset_{T}\}\subset G$
is a family of compact subsets satisfying that $\haar_{G}\brac{\Gset_{T}}\to\infty$
as $T\to\infty$. If the family $\{\Gset_{T}\}$ is Lipschitz well
rounded with parameters $(C_{\Gset},T_{0})$, then there exists $T_{1}>0$
such that for every $\delta>0$ and $T>T_{1}$:
\[
\#\left|\brac{\Gset_{T}\cap\Lat}-\frac{\haar\brac{\Gset_{T}}}{\haar\brac{\Lat\backslash G}}\right|=O_{\Lat,\delta}\left(C_{\Gset}^{\frac{\dim G}{1+\dim G}}\cdot\haar\brac{\Gset_{T}}^{1-\errexp\left(\Lat\right)+\delta}\right),
\]
where $\errexp\left(\Lat\right)$ is as in (\ref{eq: tau}). The
parameter $T_{1}$ is such that $T_{1}\geq T_{0}$ and for every $T\geq T_{1}$
\begin{equation}
C_{\Gset}^{\frac{\dim G}{1+\dim G}}=O_{\Gamma}\brac{\haar\left(\Gset_{T}\right)^{\errexp\left(\Lat\right)}}.\label{eq: def of T_1 in GN thm}
\end{equation}
\end{thm}

\begin{proof}[Proof of Proposition \ref{thm: Counting with S and W}]
According to Theorem \ref{thm: GN Counting thm}, the proposition
follows once showing that the sets $\trunc{\byparby{\funddom}T}{\Svec}\brac{\pairset}$
are Lipschitz well rounded. Recall that $\wc$ is diffeomorphic to
the group $K\times P^{\dprime}$ where $P^{\dprime}$ is, in turn,
diffeomorphic to $A^{\dprime}\times N^{\dprime}$; therefore
\[
\trunc{\byparby{\funddom}T}{\Svec}\brac{\pairset}=\trunc{\wc_{\pairset}}{\Svec}\,\byparby{A^{\prime}}TN_{\cube}^{\prime}\simeq\trunc{\brac{K\times A^{\dprime}\times N^{\dprime}}_{\pairset}}{\Svec}\,\byparby{A^{\prime}}TN_{\cube}^{\prime}\,.
\]
In \cite{HK_WellRoundedness}, we developed a method to consider the
well roundedness of families of sets of that form. We have shown \cite[Cor.~4.3]{HK_WellRoundedness}
that in this type of sets, it is sufficient that each of the components
in the product (i.e. the appropriate subsets of $K$, $A^{\dprime}$,
$N^{\dprime}$, $A^{\prime}$ and $N^{\prime}$) is well rounded.
Then the product sets are well rounded, and the Lipschitz constant
$C_{\Gset}$ is the product of the Lipschitz constants of the components,
times a constant that depends on the specific decomposition of the
group, which is in this case is the Refined Iwasawa decomposition.
We have shown in \cite[Lem.~10.8]{HK_gcd} that the Refined Iwasawa
constant is $e^{2\sumS}$. For sets that are fixed,  being boundary
controllable and bounded is sufficient for well roundedness with Lipschitz
constant that depends on the set (\cite[Prop.~3.5]{HK_WellRoundedness}).
Hence $N_{\cube}^{\prime}$ is Lipschitz well rounded with Lipschitz
constant in $O(1)$. The case of $\trunc{\wc_{\pairset}}{\Svec}$
is more complicated: it is boundary controllable (by Lemma \ref{prop: spread model BCS})
and bounded, but $\wc$ is not a group. However, it is diffeomorphic
to a product of groups, and indeed $\trunc{\brac{K\times A^{\dprime}\times N^{\dprime}}_{\pairset}}{\Svec}$
is Lipschitz well rounded with Lipschitz constant in $O(1)$, independently
of $\Svec$ (\cite[Lem.~11.1]{HK_gcd}). Hence, if we assume for now
that the families $\cbrac{\parby{A^{\prime}}T}$ and $\cbrac{\bypar{A^{\prime}}T}$
are also Lipschitz well rounded with Lipschitz constants that are
$O(1)$, then the families $\trunc{\byparby{\funddom}T}{\Svec}\brac{\pairset}$
are Lipschitz well rounded, with Lipschitz constant $C=O\brac{e^{2\sumS}}$.
In particular, 
\[
C^{\frac{1}{2\lm_{n}}}=C^{\frac{\dim\sl n\brac{\RR}}{1+\dim\sl n\brac{\RR}}}=O\brac{e^{\sumS/\lm_{n}}}.
\]
 It therefore remains to verify that the families $\cbrac{\byparby{A^{\prime}}T}\subset A^{\prime}$
are Lipschitz well rounded with Lipschitz constants that are $O(1)$.
For the family $\cbrac{\parby{A^{\prime}}T}$, this has been proved
in \cite[Prop.~9.6]{HK_gcd}. For the family $\cbrac{\bypar{A^{\prime}}T}$,
we recall Lemma \ref{lem: Florian measure} and the definition of
$f_{m}(T)$. Let 
\begin{align*}
f_{m}(T;a) & :=\int_{a}^{T}e^{s_{m}}\int_{a}^{T-s_{m}}e^{s_{m-1}}\cdots\int_{a}^{T-s_{2}-\cdots-s_{m}}e^{s_{1}}ds_{1}\cdots ds_{m}.\\
 & =\int_{0}^{T-a}e^{s_{m}-a}\int_{0}^{T-a-s_{m}}e^{s_{m-1}-a}\cdots\int_{0}^{T-a-s_{2}-\cdots-s_{m}}e^{s_{1}-a}ds_{1}\cdots ds_{m}\\
 & =e^{-ma}f_{m}(T-a).
\end{align*}
Then, by Lemma \ref{lem: Florian measure},\setlength{\abovedisplayskip}{0pt} 
\setlength{\belowdisplayskip}{0pt}
\[
f_{m}(T;a)=e^{-ma}f_{m}(T-a)=e^{T-(m+1)a}\cdot\sum_{i=0}^{m-1}(-1)^{m-i-1}\frac{(T-a)^{i}}{i!}+(-e^{-a})^{m}.
\]
Recall that 
\[
\bypar{A^{\prime}}T=\left\{ a_{t_{1},\dots,t_{\leng-1}}^{\prime}=\exp\left(\sum_{i=1}^{\leng-1}t_{i}\liev_{i}^{\prime}\right):t_{1},\dots,t_{\leng-1}\geq0;\,\sum_{i=1}^{\leng-1}(d_{i}+d_{i+1})t_{i}\leq T\right\} ,
\]
where
\[
\liev_{i}^{\prime}=(\underset{d_{i-1}}{\underbrace{0,\dots,0}},\underset{d_{i}}{\underbrace{d_{i}^{-1},\dots,d_{i}^{-1}}},\underset{d_{i+1}}{\underbrace{-d_{i+1}^{-1},\dots,-d_{i+1}^{-1}}},0,\dots,0)\in\lie{A^{\prime}}.
\]
Then
\begin{align*}
\bypar{A^{\prime\,\left(+\e\right)}}T & =\left\{ \exp\left(\sum_{i=1}^{\leng-1}(t_{i}+s_{i})\liev_{i}^{\prime}\right):\substack{t_{1},\dots,t_{\leng-1}\geq0\\
\,s_{1},\dots,s_{\leng-1}\in[-\epsilon,\epsilon]
}
;\,\sum_{i=1}^{\leng-1}(d_{i}+d_{i+1})t_{i}\leq T\right\} \\
 & \subset\left\{ \exp\left(\sum_{i=1}^{\leng-1}t_{i}\liev_{i}^{\prime}\right):t_{1},\dots,t_{\leng-1}\geq-\epsilon;\,\sum_{i=1}^{\leng-1}(d_{i}+d_{i+1})t_{i}\leq T+(2n-d_{1}-d_{\leng})\epsilon\right\} \\
 & \subset\left\{ \exp\left(\sum_{i=1}^{\leng-1}t_{i}\liev_{i}^{\prime}\right):t_{1},\dots,t_{\leng-1}\geq-2n\epsilon;\,\sum_{i=1}^{\leng-1}(d_{i}+d_{i+1})t_{i}\leq T+2n\epsilon\right\} \\
 & =:B^{(+\epsilon)}
\end{align*}
and
\begin{align*}
\bypar{A^{\prime\,\left(-\e\right)}}T & =\left\{ \exp\left(\sum_{i=1}^{\leng-1}t_{i}\liev_{i}^{\prime}\right):\forall s_{1},\dots,s_{\leng-1}\in[-\epsilon,\epsilon]:\substack{t_{1}+s_{1},\dots,t_{\leng-1}+s_{\leng-1}\geq0,\\
\sum_{i=1}^{\leng-1}(d_{i}+d_{i+1})(t_{i}+s_{i})\leq T
}
\right\} \\
 & \subset\left\{ \exp\left(\sum_{i=1}^{\leng-1}t_{i}\liev_{i}^{\prime}\right):t_{1},\dots,t_{\leng-1}\geq\epsilon;\,\sum_{i=1}^{\leng-1}(d_{i}+d_{i+1})t_{i}\leq T-(2n-d_{1}-d_{\leng})\epsilon\right\} \\
 & \subset\left\{ \exp\left(\sum_{i=1}^{\leng-1}t_{i}\liev_{i}^{\prime}\right):t_{1},\dots,t_{\leng-1}\geq2n\epsilon;\,\sum_{i=1}^{\leng-1}(d_{i}+d_{i+1})t_{i}\leq T-2n\epsilon\right\} \\
 & =:B^{(-\epsilon)}.
\end{align*}
As a result, for $\epsilon<\brac{4\left(\leng-1\right)n}^{-1}$,
\[
\frac{\mu_{A^{\prime}}\left(\bypar{A^{\prime\,\left(+\e\right)}}T\right)}{\mu_{A^{\prime}}\left(\bypar{A^{\prime\,\left(-\e\right)}}T\right)}\leq\frac{\mu_{A^{\prime}}\left(B^{(+\epsilon)}\right)}{\mu_{A^{\prime}}\left(B^{(-\epsilon)}\right)}=\frac{f_{\leng-1}(T+2n\epsilon;-2n\epsilon)}{f_{\leng-1}(T-2n\epsilon;+2n\epsilon)}=
\]
\[
=\frac{e^{2\left(\leng-1\right)n\epsilon}f_{\leng-1}(T)}{e^{-2\left(\leng-1\right)n\epsilon}f_{\leng-1}(T)}=e^{4n\left(\leng-1\right)\epsilon}\leq1+\frac{\epsilon}{6\left(\leng-1\right)n}\,.\qedhere
\]
\end{proof}
The proof of Proposition \ref{thm: Counting with S and W} relies
on the fact that the sets $\trunc{\funddom_{T}}{\Svec}\brac{\pairset}$
and $\trunc{\bypar{\funddom}T}{\Svec}\brac{\pairset}$ are Lipschitz
well rounded, which only happens when $\Svec$ is fixed. In the following
claim, we will extend the counting in Proposition \ref{thm: Counting with S and W}
to the case where 
\[
\Svec=\Svec\brac T=\brac{\Flat{S_{1}}1\left(T\right),\ldots\ldots,\Flat{S_{d_{\leng}-1}}{\leng}\left(T\right)}
\]
grows as $T$ does. The sets $\trunc{\byparby{\funddom}T}{\Svec(T)}\brac{\pairset}$
are not well rounded, hence the idea of the proof would be to control
the growth of $\Svec\brac T$ such that the number of lattice points
in $\byparby{\funddom}T\brac{\pairset}-\trunc{\byparby{\funddom}T}{\Svec(T)}(\pairset)$
would get swallowed in the error term obtained in Proposition \ref{thm: Counting with S and W}.
\begin{prop}
\label{thm: Counting with S(T) and W(T)}With the notations from Proposition
\ref{thm: Counting with S and W}, $0<\e<\errexp$, $\delta\in\left(0,\errexp-\e\right)$,
and $\Svec\left(T\right)$ such that $\sumS\left(T\right)\leq\dl\lm_{n}\ind T+O_{\pairset}(1)$:
\[
\#\brac{\trunc{\byparby{\funddom}T}{\Svec(T)}\brac{\pairset}\cap\Lat}=\frac{\mu\left(\byparby{\funddom}T\brac{\pairset}\right)}{\mu\left(\Lat\backslash G\right)}+O_{\pairset,\e}\brac{e^{\ind T\left(1-\errexp+\dl+\e\right)}}.
\]
\end{prop}

\begin{proof}
We compute a bound on $\sumS\brac T$ for which the error term established
in Proposition \ref{thm: Counting with S and W} remains smaller than
the main term therein. According to Proposition \ref{prop: volumes}
, the main term in Proposition \ref{thm: Counting with S and W} is
of order $\mu\brac{\trunc{\byparby{\funddom}T}{\Svec}\brac{\pairset}}\approx e^{\ind T}$
and the error term is of order $e^{\frac{\sumS}{\lm_{n}}}e^{\ind T}$.
Namely, we require the existence of $\ga\in\left(0,1\right)$ for
which
\[
\sumS\brac T/\lm_{n}+\left(1-\errexp+\e\right)\cdot\ind T=\ga\ind T.
\]
Let $\delta$ denote the number $\ga+\errexp-\e-1$, i.e.\ $\ga=\delta+1+\e-\errexp$.
Then $\ga<1$ if and only if $\dl<\errexp-\e$, where $\errexp-\e$
is positive since $\errexp>\e$. We conclude that for $0<\dl<\errexp-\e$
and $\sumS\left(T\right)<\dl\lm_{n}\ind T$, the counting in $\trunc{\byparby{\funddom}T}{\Svec(T)}\brac{\pairset}$
applies with an error term of order $e^{\ga\ind T}=e^{T\ind\left(1-\errexp+\dl+\e\right)}$,
and main term that is $\mu\brac{\trunc{\byparby{\funddom}T}{\Svec(T)}\brac{\pairset}}\approx\mu\brac{\byparby{\funddom}T\brac{\pairset}}$
(as, according to (\ref{eq: volumes truncated}), the difference in
volumes between $\trunc{\byparby{\funddom}T}{\Svec(T)}\brac{\pairset}$
and $\byparby{\funddom}T\brac{\pairset}$ is swallowed in the error
term). 

As for the lower bound $T_{1}$ on $T$, in Proposition \ref{thm: Counting with S and W}
we had $\sumS\leq\ind\lm_{n}\errexp T+O_{\pairset}\left(1\right)$;
hence, combining both bounds on $\sumS$ we obtain
\[
\sumS\leq\min\left\{ \ind\lm_{n}\dl T,\:\ind\lm_{n}\errexp T\right\} +O_{\pairset}\left(1\right)=\ind\lm_{n}\dl T+O_{\pairset}\left(1\right)
\]
for $T$ large enough and $\delta\in\left(0,\errexp-\e\right)$. 
\end{proof}

\section{\label{sec: Counting the cusp}Neglecting the cusp}

The goal of this final section is to complete the counting of $\sl n(\ZZ)$
elements in $\funddom_{T}$ and $\bypar{\funddom}T$, by counting
in the sets $\funddom_{T}-\trunc{\funddom_{T}}{\Svec}$ and $\bypar{\funddom}T-\,\trunc{\bypar{\funddom}T}{\Svec}$
as $\Svec$ grows linearly with $T$. We prove the following:
\begin{prop}
\label{cor: very few SL(n,Z) points up the cusp}Let 
\[
\underline{\sigma}=\brac{\Flat{\s_{1}}1,\ldots,\Flat{\s_{d_{1}-1}}1,\dots,\Flat{\s_{1}}{\leng},\ldots,\Flat{\s_{d_{\leng}-1}}{\leng}}
\]
where $0<\Flat{\s_{i}}j<1$ for all $i,j$, and $\sigma_{\min}=\min_{i,j}\Flat{\s_{i}}j$.
Then for every $\e>0$
\begin{align*}
\#\left|\left(\funddom_{T}-\trunc{\funddom_{T}}{\underline{\sigma}T}\right)\cap\sl n\left(\ZZ\right)\right| & =O_{\e}\left(e^{T\left(2n-d_{l}-d_{1}-\sigma_{\min}+\e\right)}\right).
\end{align*}
\end{prop}

\begin{prop}
\label{prop:Florian very few SL(n,Z) points up the cusp}For $\underline{\s}$
and $\sigma_{\min}$ as in Proposition \ref{prop:Florian very few SL(n,Z) points up the cusp}
and every $\e>0$
\begin{align*}
\left|\left(\bypar{\funddom}T-\,\trunc{\bypar{\funddom}T}{\underline{\sigma}T}\right)\cap\sl n\left(\ZZ\right)\right| & =O_{\e}\left(e^{T\left(1-\sigma_{\min}+\e\right)}\right).
\end{align*}
\end{prop}

The proofs of Propositions \ref{cor: very few SL(n,Z) points up the cusp}
and \ref{prop:Florian very few SL(n,Z) points up the cusp} require
the concept of a reduced basis, which was introduced in the context
of the construction of Siegel sets (e.g. \cite[ X]{Bekka_Mayer,Rag72},
\cite[Sec.~4.4 (4.23)]{Terras}). Recall that if an $n\times n$ matrix
has columns $(v_{1},\ldots,v_{n})$ and Iwasawa coordinates $kan$
where $k=(\phi_{1},\ldots,\phi_{n})$ and $a=\diag{a_{1},\ldots,a_{n}}$,
then for every $j$ the vector $a_{j}\phi_{j}$ is the projection
of $v_{j}$ to the space $V_{j-1}^{\perp}=\perpen{\brac{\sp{\RR}{v_{1},\ldots,v_{j-1}}}}$,
and in particular $a_{j}$ is the distance of $v_{j}$ from $V_{j-1}=\sp{\RR}{v_{1},\ldots,v_{j-1}}=\sp{\RR}{\phi_{1},\ldots,\phi_{j-1}}$. 
\begin{defn}[{\cite[Def.~3.4]{HK_gcd}}]
\label{def: Siegel reduced basis}A basis $\left\{ v_{1},\ldots,v_{\dd}\right\} $
for a lattice $\lat$ is called \emph{reduced} if for every $j\in\left\{ 1,\ldots,\dd\right\} $
the element $v_{j}$ satisfies that:

\begin{enumerate}
\item \textbf{(red1)} The length $a_{j}\neq0$ of the projection $a_{j}\phi_{j}$
of $v_{j}$ to $V_{j-1}^{\perp}$ is minimal (where $V_{0}=\sp{}{\emptyset}=\left\{ 0\right\} $).
\item \textbf{(red2)} The projection of $v_{j}$ to $V_{j-1}$ lies in the
Dirichlet domain of the lattice $\sp{\ZZ}{a_{1}\phi_{1},\dots,a_{j}\phi_{j}}$.
\end{enumerate}
\end{defn}

If the columns of a matrix form a reduced basis to the lattice they
span, and the matrix has Iwasawa coordinates $kan$, then $a_{j}\leq a_{j+1}$
for every $j$, and the above-diagonal entries of the upper unipotent
matrix $n$ lie in $\sbrac{-1/2,1/2}$.

The fundamental domain $\groupfund n\subset\sl n\brac{\RR}$ defined
in (\ref{eq: fund SL(Z)}) satisfies that if $g\in\groupfund n$ then
the columns of $g$ form a reduced basis to the lattice that they
span. We shall now extend the definition of a reduced basis from
lattices to flags, so that if $g\in\sl n\brac{\RR}$ is in $\diag{\groupfund{d_{1}},\groupfund{d_{2}},\ldots,\groupfund{d_{\leng}}}\subset G^{\dprime}$,
then the columns of $g$ form a reduced basis to the $\dvec$-flag
that they span, $\flag_{g}\brac{\ZZ}$.
\begin{defn}
\label{def: Siegel reduced basis for flags}Let $\Lambda<\RR^{n}$
be an $\dd$-lattice, and let $d_{1},\dots,d_{\leng-1}$ be a partition
of $\dd$. Recall $D_{j}=d_{1}+\cdots+d_{j}$ (where $D_{0}=0$).
A basis $v_{1},\dots,v_{\dd}$ for $\Lambda$ is called $\brac{d_{1},\dots,d_{\leng-1}}$-reduced
if:
\begin{enumerate}
\item It satisfies \textbf{(red2)}.
\item The vectors $v_{D_{j-1}+1},\dots,v_{D_{j}}$ satisfy \textbf{(red1)}
 for every $j=1,\dots,\leng-1$.
\end{enumerate}
A basis for a primitive $\dvec$-flag of lattices $\flag\brac{\ZZ}$,
where $\dvec=(d_{1},\ldots,d_{\leng-1},d_{\leng})$, is called reduced
if it its first $n-d_{\leng}$ elements form a $\brac{d_{1},\dots,d_{\leng-1}}$-reduced
basis to the lattice that they span. 
\end{defn}

The following lemma will play a key role. 
\begin{lem}
\label{lem: Butz's Lemma}Let $\latfull<\RR^{n}$ be a full lattice,
and let $\left[\a_{i},\b_{i}\right]$ be $\dd-1$ intervals with $0\leq\a_{i}<\b_{i}$.
The number of $r$-lattices $\lat<\latfull$ satisfying that 
\[
\covol{\lat^{i}}\in\sbrac{\cov^{\a_{i}},\cov^{\b_{i}}},
\]
where $\lat^{i}=\sp{\ZZ}{v_{1},\ldots,v_{i}}$ and $\cbrac{v_{1},\ldots,v_{\dd}}$
is some $\brac{d_{1},\dots,d_{\leng-1}}$-reduced basis of $\lat$,
is $O_{n,\covol{\latfull}}\brac{\cov^{e\left(\underline{\a},\underline{\b}\right)}}$,
where 
\[
e\brac{\underline{\a},\underline{\b}}=2\sum_{i=1}^{\dd}\b_{i}+(n-\dd-1)\beta_{\dd}+\sum_{i=1}^{\dd-1}\brac{n-i}\brac{\b_{i}-\a_{i}}.
\]
\end{lem}

In fact, for this lemma it is sufficient that the bases $\cbrac{v_{1},\ldots,v_{\dd}}$
satisfy (\textbf{red2}). A special case of Lemma \ref{lem: Butz's Lemma}
in which $\a_{\dd}=0$ and $\b_{\dd}=1$ appears in \cite[Lem.~6.3]{HK_gcd}.
The proofs are actually identical, but we prove the lemma here for
completeness.
\begin{proof}
We count the number of possibilities to choose a $\brac{d_{1},\dots,d_{\leng-1}}$-reduced
basis $\cbrac{v_{1},\ldots,v_{\dd}}$ for $\lat$, which satisfies
$\covol{\lat^{i}}\in\sbrac{\cov^{\a_{i}},\cov^{\b_{i}}}$ for every
$i=1,\ldots,\dd$. Recall that $a_{i}$ is the distance of $v_{i}$
from the subspace $V_{i-1}$, which means $a_{i}=\covol{\lat^{i}}/\covol{\lat^{i-1}}$.
As a result, if $\lat$ is such that $\covol{\lat^{i}}\in\left[\cov^{\a_{i}},\cov^{\b_{i}}\right]$,
then 
\[
a_{i}\leq R_{i}:=X^{\b_{i}-\a_{i-1}}.
\]
Denote by $\#v_{i}\vert_{\lat^{i-1}}$ the number of possibilities
for choosing $v_{i}$ given that $\lat^{i-1}$ is known. We first
claim that for every $1\leq i\leq\dd$,
\begin{equation}
\#v_{i}\vert_{\lat^{i-1}}=O\left(\brac{R_{i}}^{n-i+1}\cdot\covol{\lat^{i-1}}\right).\label{eq: =000023 of possibilities for v_i}
\end{equation}
Indeed, $\#v_{1}\vert_{\lat^{0}}$ is simply the number of possibilities
for choosing an element $v_{1}$ of  $\Flat{\lat}{\leng}$ inside
an origin-centered ball in $\RR^{n}$ of radius $a_{1}=\norm{v_{1}}\leq R_{1}$,
namely
\[
\#v_{1}\vert_{\lat^{0}}\le\#\brac{\Flat{\lat}{\leng}\cap\ball{R_{1}}^{n}}=O\brac{R_{1}^{n}}.
\]
For $i>1$, recall that the orthogonal projection of $v_{i}$ to the
subspace $V_{i-1}$ lies inside a Dirichlet domain of the lattice
\[
\tilde{\lat}^{i-1}:=\sp{\ZZ}{a_{1}\phi_{1},\dots,a_{i-1}\phi_{i-1}}.
\]
Thus, $v_{i}$ has to be chosen from the set of $\Flat{\lat}{\leng}$
elements that are of distance at most $a_{i}\leq R_{i}$ from the
Dirichlet domain for $\tilde{\lat}^{i-1}$ in $\sp{\RR}{\lat^{i-1}}$.
These are the $\Flat{\lat}{\leng}$ elements that lie in the product
of the Dirichlet domain for $\tilde{\lat}^{i-1}$ (in $V_{i-1}$)
with an origin-centered ball $\ball{R_{i}}^{n-\left(i-1\right)}$
in the $n-(i-1)$ dimensional subspace $V_{i-1}^{\perp}$, of radius
$R_{i}$. Then 
\begin{eqnarray*}
\#v_{i}\vert_{\lat^{i-1}} & \leq & \#\brac{\Flat{\lat}{\leng}\cap\cbrac{\ball{R_{i}}^{n-\left(i-1\right)}\times\mbox{Dirichlet domain for \ensuremath{\tilde{\lat}^{i-1}}}}}\\
 & = & O\brac{\vol\brac{\ball{R_{i}}^{n-\left(i-1\right)}}\cdot\covol{\lat^{i-1}}}\\
 & = & O\brac{R_{i}^{n-i+1}\cdot\covol{\lat^{i-1}}},
\end{eqnarray*}
which proves (\ref{eq: =000023 of possibilities for v_i}). Now, the
number of possibilities for $\Flat{\lat}1<\cdots<\Flat{\lat}{\leng-1}=\lat$
is
\[
O\left(\prod_{i=1}^{\dd}\brac{\#v_{i}\vert_{\lat^{i-1}}}\right)=O\left(\prod_{i=1}^{\dd}\brac{R_{i}^{n-i+1}\cdot\covol{\lat^{i-1}}}\right)
\]
\[
=O\left(\prod_{i=1}^{\dd}\brac{X^{\brac{\b_{i}-\a_{i-1}}\brac{n-i+1}}\cdot X^{\b_{i-1}}}\right),
\]
where $\a_{0}=0$ (as $\covol{\lat^{1}}=\left\Vert v_{1}\right\Vert \geq\cov^{0}$).
Since 
\begin{align*}
\sum_{i=1}^{\dd}\brac{n-i+1}\brac{\b_{i}-\a_{i-1}}+\b_{i-1} & =\\
\sum_{i=1}^{\dd-1}\brac{n-i}\brac{\b_{i}-\a_{i}}+2\sum_{i=1}^{\dd-1}\b_{i}+(n-d+1)\beta_{\dd} & =e\brac{\underline{\a},\underline{\b}},
\end{align*}
then the number of such lattices $\lat$ is bounded by $X^{e\brac{\underline{\a},\underline{\b}}}$. 
\end{proof}
\begin{cor}
\label{cor: very few lattices with very short vector}Assume the notations
of Lemma \ref{lem: Butz's Lemma}, and let $0<\theta_{1},\dots,\theta_{\dd}<1$,
$\boldsymbol{\theta}=\sum_{i=1}^{\dd}\theta_{i}$. The number of integral
$\dd$-lattices $\lat<\ZZ^{n}$ satisfying that 
\[
\covol{\lat^{i}}\in\sbrac{1,\cov^{\theta_{i}}}
\]
is, for every $\e>0$, 
\[
O_{\e,C}\,\brac{\cov^{(n-\dd-1)\theta_{\dd}+2\boldsymbol{\theta}+\e}}.
\]
\end{cor}

\begin{proof}
The proof is identical to the one in \cite[Prop.~6.4]{HK_gcd}.
\end{proof}
\begin{proof}[Proof of Proposition \ref{cor: very few SL(n,Z) points up the cusp}]
Let $\ga_{\flag}=ka_{\underline{s}}^{\prime\prime}a_{\underline{t}}^{\prime}n\in\sl n\left(\ZZ\right)$.
Recall that $\ga_{\flag}\in\funddom_{T}-\funddom_{T}^{\underline{\sigma}T}$
if and only if the columns of $\ga_{\flag}$ form a basis to $\flag_{g}\brac{\ZZ}$,
and there exist $j\in\{1,\dots,\leng\}$ and $\ii\in\left\{ 1,\ldots,d_{j}-1\right\} $
such that $\Flat{s_{\ii}}j>\Flat{\sigma_{\ii}}jT$ w.r.t.\ the reduced
basis in the columns of $\ga_{\flag}$. 

Assume first that $j<\leng$, and let $i=D_{j-1}+\ii$. Set $\lat=\Flat{\lat}{\leng-1}$,
whose rank is $\dd=n-d_{\leng}$. By Proposition \ref{prop: explicit RI coordinates of g}$(iii)$
and the fact that $\lat^{i}$ is integral, 
\[
1\leq\covol{\lat^{\jj}}=\covol{\Flat{\lat}{j-1}}\covol{L_{j}^{\ii}}=e^{t_{j-1}}\cdot e^{\frac{\ii\left(t_{j}-t_{j-1}\right)}{d_{j}}-\frac{\Flat{s_{\ii}}j}{2}}
\]
\[
\leq e^{\frac{\ii t_{j}+(d_{j}-\ii)t_{j-1}}{d_{j}}-\frac{\Flat{\sigma_{\ii}}jT}{2}}\leq\begin{cases}
e^{T\brac{\frac{\ii}{d_{1}}-\frac{\Flat{\sigma_{\ii}}1}{2}}} & j=1\\
e^{T\brac{1-\frac{\Flat{\sigma_{\ii}}j}{2}}} & j>1
\end{cases}.
\]
Similar considerations show that for all other pairs $(x,\jjp)$
\textcolor{orange}{} with $j<\leng$:
\[
1\leq\covol{\lat^{\jjp}}\leq e^{\frac{\iip t_{x}+(d_{x}-\iip)t_{x-1}}{d_{x}}}\leq\begin{cases}
e^{T\cdot\frac{\iip}{d_{1}}} & x=1\\
e^{T} & x>1
\end{cases}.
\]
We now apply Corollary \ref{cor: very few lattices with very short vector}
for the lattice $\lat$, with $\t_{1},,,,\t_{n-d_{\leng}}$ that stem
from the inequalities above; note also that 
\[
\t_{d_{1}}=\t_{d_{1}+d_{2}}=\cdots=\t_{d_{1}+\cdots+d_{\leng-1}=n-d_{\leng}}=1,
\]
since for every $k=1,\ldots,\leng-1$ we have that 
\[
\covol{\Flat{\lat}k}=\covol{\lat^{d_{1}+\cdots+d_{k}}}=e^{t_{k}}\leq e^{T}=X.
\]
The application of Corollary \ref{cor: very few lattices with very short vector}
in the case where the $j$ for which $\Flat{s_{\ii}}j>\Flat{\sigma_{\ii}}jT$
is $j=1$ is with the following $\t_{1},,,,\t_{n-d_{\leng}}$: 
\[
\begin{array}{ccc}
\brac{\theta_{1},\ldots,\theta_{d_{1}}} & = & \brac{\frac{1}{d_{1}},\ldots,\frac{\ii}{d_{1}}-\frac{\Flat{\sigma_{\ii}}1}{2},\dots,\frac{d_{1}}{d_{1}}}\\
\brac{\theta_{d_{1}+1},\ldots,\theta_{d_{1}+d_{2}}} & = & \brac{1,\ldots,1}\\
\vdots\\
\brac{\theta_{d_{1}+\cdots+d_{\leng-2}+1},\ldots,\theta_{d_{1}+\cdots+d_{\leng-2}+d_{\leng-1}}} & = & \brac{1,\ldots,1}
\end{array},
\]
for which the error exponent in Corollary \ref{cor: very few lattices with very short vector}
is 
\begin{align*}
(d_{\leng}-1)\theta_{n-d_{\leng}}+2\sum\theta_{i} & =d_{\leng}-1+2\brac{n-d_{\leng}-d_{1}+\frac{1}{d_{1}}\sum_{k=1}^{d_{1}}k-\frac{2\Flat{\sigma_{\ii}}1}{2}}\\
 & =2n-d_{1}-d_{\leng}-\Flat{\sigma_{\ii}}1.
\end{align*}
The application of Corollary \ref{cor: very few lattices with very short vector}
in the case where the $j$ for which $\Flat{s_{\ii}}j>\Flat{\sigma_{\ii}}jT$
is $1<j<\leng$ is with the following $\t_{1},,,,\t_{n-d_{\leng}}$:
\[
\begin{array}{ccc}
\brac{\theta_{1},\ldots,\theta_{d_{1}}} & = & \brac{\frac{1}{d_{1}},\ldots,\frac{d_{1}}{d_{1}}}\\
\brac{\theta_{d_{1}+1},\ldots,\theta_{d_{1}+d_{2}}} & = & \brac{1,\ldots,1}\\
\vdots & \vdots & \vdots\\
\brac{\theta_{d_{1}+\cdots+d_{j-1}+1},\ldots,\theta_{d_{1}+\cdots+d_{j-1}+d_{j}}} & = & \brac{1,\ldots,1-\frac{\Flat{\sigma_{\ii}}j}{2},\dots,1}\\
\vdots & \vdots & \vdots\\
\brac{\theta_{d_{1}+\cdots+d_{\leng-2}+1},\ldots,\theta_{d_{1}+\cdots+d_{\leng-2}+d_{\leng-1}}} & = & \brac{1,\ldots,1}
\end{array},
\]
for which the error exponent in Corollary \ref{cor: very few lattices with very short vector}
is
\begin{align*}
(d_{\leng}-1)\theta_{n-d_{\leng}}+2\sum\theta_{i} & =d_{\leng}-1+2\left(n-d_{\leng}-d_{1}+\frac{1}{d_{1}}\sum_{z=1}^{d_{1}}z-\frac{\s_{\ii}^{(j)}}{2}\right)\\
 & =2n-d_{1}-d_{\leng}-\Flat{\sigma_{\ii}}j.
\end{align*}

We see that for any $1\leq j<\leng$ one has that $(d_{l}-1)\theta_{d}+2\sum\theta_{i}=2n-d_{1}-d_{\leng}+1-\s_{\ii}^{(j)}$,
hence according to Corollary \ref{cor: very few lattices with very short vector},
the number of such possible flags is 
\[
O_{\e}\brac{e^{T\left(2n-d_{1}-d_{\leng}-\s_{\ii}^{(j)}+\e\right)}}.
\]

Finally, assume that $j=\leng$. By \cite[Prop.~2.2 and Lem.~A.13]{HK_dlattices},
since $\Flat{\lat}{\leng-1}$ is primitive, 
\[
\frac{\covol{\brac{\perpen{\brac{\Flat{\lat}{\leng-1}}}}^{d_{\leng}-\ii}}}{\covol{\Flat{\lat}{\leng-1}}}\asymp\covol{\qlat_{\leng}^{\ii}}.
\]
The right-hand side is in fact $\covol{\qlat_{\leng}^{\ii}}=e^{\frac{\ii-t_{\leng-1}}{d_{j}}-\frac{\Flat{s_{\ii}}{\leng}}{2}}$
by \ref{prop: explicit RI coordinates of g}$(iii)$, while $\covol{\Flat{\lat}{\leng-1}}=e^{t_{\leng-1}}$.
We get that, up to an additive constant that becomes negligible when
$t_{\leng-1}$ is large, $\Flat{s_{\ii}}{\leng}>\Flat{\sigma_{\ii}}{\leng}T$
implies that
\[
1\leq\covol{\brac{\lat^{\perp}}^{d_{\leng}-\ii}}<e^{t_{\leng-1}-\frac{\ii t_{\leng-1}}{d_{\leng}}-\frac{\Flat{\sigma_{\ii}}{\leng}t_{\leng-1}}{2}}\leq e^{T\brac{\frac{d_{\leng}-\ii}{d_{\leng}}-\frac{\Flat{\sigma_{\ii}}{\leng}}{2}}}.
\]
Now consider the flag $\{0\}<\perpen{\brac{\Flat{\lat}{\leng-1}}}<\perpen{\brac{\Flat{\lat}{\leng-2}}}<\cdots<\perpen{\brac{\Flat{\lat}1}}<\ZZ^{n}$
(which clearly determine the original flag); by the same considerations
as for the case $j=1$ above, the number of such possible lattice
flags is $e^{T\left(2n-d_{1}-d_{\leng}-\Flat{\sigma_{\ii}}{\leng}+\e\right)}$.
All in all, 
\[
\#\left|\brac{\parby{\funddom}T-\funddom_{T}^{\underline{\sigma}T}}\cap\sl n\left(\ZZ\right)\right|=O\brac{e^{T\left(2n-d_{1}-d_{\leng}-\sigma_{\min}+\e\right)}}.\qedhere
\]
\end{proof}
\begin{proof}[Proof of Proposition \ref{prop:Florian very few SL(n,Z) points up the cusp}]
Let $\ga_{\flag}=ka_{\underline{s}}^{\prime\prime}a_{\underline{t}}^{\prime}n\in\sl n\left(\ZZ\right)$,
and assume that
\[
\flag\brac{\ZZ}=\Flat{\lat}0<\cdots<\Flat{\lat}{\leng-1}<\Flat{\lat}{\leng}=\ZZ^{n}.
\]
By definition, $\ga_{\flag}\in\,\bypar{\funddom}T-\,\bypar{\funddom}T^{\underline{\sigma}T}$
if and only if (i) the columns of $\ga_{\flag}$ form a basis to $\flag_{g}\brac{\ZZ}$,
(ii) there are $0\leq T_{1},\dots,T_{\leng-1}$ such that 
\[
\log\brac{\covol{\Flat{\lat}j}}\leq T_{j}
\]
and
\[
\sum_{j=1}^{\leng-1}(d_{j}+d_{j+1})T_{j}\leq T,
\]
and (iii) there exist $j\in\{1,\dots,\leng\}$ and $\jj=D_{j-1}+\ii$
with $\ii\in\left\{ 1,\ldots,d_{j}-1\right\} $ for which $\Flat{s_{\ii}}j\geq\Flat{\sigma_{\ii}}jT$.
Set $\lat=\Flat{\lat}{\leng-1}$, whose rank is $\dd=n-d_{\leng}$.
Then, if $j<\leng$, Proposition \ref{prop: explicit RI coordinates of g}$(iii)$
and the fact that $\lat^{\jj}$ is integral imply that
\[
1\leq\covol{\lat^{\jj}}=\covol{\Flat{\lat}{j-1}}\covol{L_{j}^{\ii}}=e^{t_{j-1}}\cdot e^{\frac{\ii\left(t_{j}-t_{j-1}\right)}{d_{j}}-\frac{\Flat{s_{\ii}}j}{2}}
\]
\[
\leq e^{\frac{\ii t_{j}+(d_{j}-\ii)t_{j-1}}{d_{j}}-\frac{\Flat{\sigma_{\ii}}jT}{2}}\leq\begin{cases}
e^{T_{1}\frac{\ii}{d_{j}}-T\frac{\Flat{\sigma_{\ii}}j}{2})} & j=1\\
e^{T_{j-1}\frac{d_{j}-\ii)}{d_{j}}+T_{j}\frac{\ii}{d_{j}}-T\frac{\Flat{\sigma_{\ii}}j}{2}} & j>1
\end{cases},
\]
where for all other pairs $(x,\jjp)$ we similarly have that
\[
1\leq\covol{\lat^{\jjp}}\leq\begin{cases}
e^{T_{1}\cdot\frac{\iip}{d_{1}}} & x=1\\
e^{\frac{(d_{x}-\iip)T_{x-1}}{d_{x}}+\frac{\iip T_{x}}{d_{x}}} & x>1
\end{cases}.
\]

In order to pass from $\leng-1$ parameters $T_{1},\ldots,T_{\leng-1}$
to a single parameter $T$, we approximate the simplex 
\[
\conv\cbrac{0,(d_{1}+d_{2})e_{1},\dots,(d_{\leng-1}+d_{\leng})e_{\leng-1}}
\]
with a covering by cubes that depend on a parameter $\dl$.  For
a fixed $\delta>0$, cover the simplex by $O_{\delta,\dd}(1)$ cubes
\[
\mathcal{C}=\prod_{j=1}^{\leng-1}[0,\alpha_{j}(\mathcal{C},\dl)]
\]
satisfying that 
\begin{equation}
\sum_{k=1}^{\leng-1}\alpha_{k}(\mathcal{C},\dl)(d_{k}+d_{k+1})\leq1+\delta.\label{eq: height cubes}
\end{equation}
For a given cube $\mathcal{C}$ and any $z=1,\dots,\leng-1$, denote
$T_{k}=T\cdot\alpha_{k}\brac{\mathcal{C},\dl}=T\a_{k}$. We now apply
Corollary \ref{cor: very few lattices with very short vector} with
\[
\begin{array}{ccc}
\brac{\theta_{1},\ldots,\theta_{d_{1}}} & = & \brac{\frac{\a_{1}}{d_{1}},\ldots,\frac{\a_{1}d_{1}}{d_{1}}}\\
\brac{\theta_{d_{1}+1},\ldots,\theta_{d_{1}+d_{2}}} & = & \brac{\frac{\a_{2}}{d_{2}}+\frac{(d_{2}-1)\a_{1}}{d_{2}},\ldots,\frac{\a_{2}d_{2}}{d_{2}}+\frac{(d_{2}-d_{2})\a_{1}}{d_{2}}}\\
\vdots\\
\brac{\theta_{d_{1}+\cdots+d_{j-1}+1},\ldots,\theta_{d_{1}+\cdots+d_{j-1}+d_{j}}} & = & (\frac{\a_{j}}{d_{j}}+\frac{(d_{j}-1)\a_{j-1}}{d_{j}},\ldots,\frac{\hat{i}\a_{j}}{d_{j}}+\frac{(d_{j}-\hat{i})\a_{j-1}}{d_{j}}-\frac{\Flat{\sigma_{\ii}}j}{2},\\
\vdots &  & ,\dots,\frac{\a_{j}d_{j}}{d_{j}}+\frac{(d_{j}-d_{j})\a_{j-1}}{d_{j}})\\
\brac{\theta_{d_{1}+\cdots+d_{\leng-2}+1},\ldots,\theta_{d_{1}+\cdots+d_{\leng-2}+d_{\leng-1}}} & = & \brac{\frac{\a_{\leng-1}}{d_{\leng-1}}+\frac{(d_{\leng-1}-1)\a_{\leng-2}}{d_{\leng-1}},\ldots,\frac{\a_{\leng-1}d_{\leng-1}}{d_{\leng-1}}+\frac{(d_{\leng-1}-d_{\leng-1})\a_{\leng-2}}{d_{\leng-1}}}
\end{array}.
\]
In particular, notice that $\t_{d_{1}+\cdots+d_{k}}=\a_{k}$ for every
$1\leq k\leq\leng-1$, reflecting the fact that $\covol{\Flat{\lat}k}\leq e^{T_{k}}=e^{\a_{k}T}=X^{\a_{k}}$.

Substituting the values of these $\t$'s into the error exponent in
Corollary \ref{cor: very few lattices with very short vector} yields
\begin{align*}
(d_{\leng}-1)\theta_{n-d_{\leng}}+2\sum\theta_{i} & =(d_{\leng}-1)\a_{\leng-1}+2\sum_{k=1}^{\leng-1}\frac{\a_{k}}{d_{k}}\sum_{x=1}^{d_{k}}x+2\sum_{k=2}^{\leng-1}\a_{k-1}\sum_{x=1}^{d_{k}}\frac{d_{k}-x}{d_{k}}-\frac{2\Flat{\sigma_{\ii}}j}{2}\\
 & =(d_{\leng}-1)\a_{\leng-1}+\sum_{k=1}^{\leng-1}\a_{k}\brac{1+d_{k}}+\sum_{k=2}^{\leng-1}\a_{k-1}\brac{d_{k}-1}-\Flat{\sigma_{\ii}}j\\
 & =\sum_{k=1}^{\leng-1}\a_{k}\brac{1+d_{k}}+\sum_{k=1}^{\leng-1}\a_{k}\brac{d_{k+1}-1}-\Flat{\sigma_{\ii}}j\\
 & =\sum_{k=1}^{\leng-1}\a_{k}\brac{d_{k}+d_{k+1}}-\Flat{\sigma_{\ii}}j.
\end{align*}
By (\ref{eq: height cubes}), the above is bounded by
\[
\leq1+\dl-\Flat{\sigma_{\ii}}j.
\]
Then, by Corollary \ref{cor: very few lattices with very short vector},
the number of such possible flags is 
\[
O_{\e}\brac{e^{T\left(1+\dl-\s_{\ii}^{(j)}+\e\right)}}.
\]
We conclude that the number of $\sl n\left(\ZZ\right)$ elements in
$\bypar{\funddom}T-\,\bypar{\funddom}T^{\underline{\sigma}T}$ is
bounded by
\[
\sum_{\mathcal{C}}\sum_{\substack{j\in\cbrac{1,\dots,\leng-1}\\
\ii\in\cbrac{1,\dots,d_{j-1}}
}
}O_{\e}\brac{e^{T\brac{1+\dl-\Flat{\sigma_{\ii}}j+\e}}},
\]
where by taking $\dl<\e$ we can conceal $\delta$ within $\epsilon$
and obtain 
\[
=O_{\e}\brac{e^{T\brac{1-\sigma_{\min}+\e}}}.
\]

The proof for the case $j=\leng$ is similar to this case in the proof
of Proposition \ref{prop: prim Z flags correspond to integral matrices}.
\end{proof}

We can now tie the edges to complete the proof of Theorem \ref{thm: general thm}.

\begin{proof}[Proof of Theorem \ref{thm: general thm}]
Let $0<\e<\errexp_{n}$, $0<\delta<\errexp_{n}-\e$ and $\ind$ as
in (\ref{eq: choose height}). Suppose first that $\pairset\subseteq\flagspace{\dvec}$
is not bounded. Recall that $\lm_{n}=\frac{n^{2}}{2\left(n^{2}-1\right)}$
and let $\underline{\sigma}=\brac{\frac{\dl\lm_{n}\ind-\e}{n-\leng}}\cdot\Onevec_{n-\leng}$.
Note that the sum of the coordinates of $\underline{\sigma}$ is $\dl\lm_{n}\ind-\e$,
which, for $T$ large enough, is smaller than $\delta\lm_{n}\ind+O(1/T)$
(aiming to satisfy the condition in Proposition \ref{thm: Counting with S(T) and W(T)}).
By Propositions \ref{cor: very few SL(n,Z) points up the cusp} (For
$H=H_{\infty}$), \ref{prop:Florian very few SL(n,Z) points up the cusp}
(for $H=H_{\ac}$) and \ref{thm: Counting with S(T) and W(T)}, we
have
\[
\#\brac{\sl n\left(\ZZ\right)\cap\byparby{\funddom}T\brac{\pairset}}=\#\brac{\sl n\left(\ZZ\right)\cap\trunc{\byparby{\funddom}T}{\underline{\sigma}T}\brac{\pairset}}+O_{\e}\brac{e^{\ind T\left(1-\frac{\dl\lm_{n}}{n-\leng}+\e\right)}}
\]
\begin{align*}
 & =\frac{\mu\brac{\byparby{\funddom}T\brac{\pairset}}}{\mu\brac{\sl n\brac{\RR}/\sl n\brac{\ZZ}}}+O_{\pairset,\e}\brac{e^{\ind T\left(1-\errexp_{n}+\delta+\e\right)}}+O_{\e}\brac{e^{\ind T\left(1-\frac{\dl\lm_{n}}{n-\leng}+\e\right)}}.
\end{align*}
We choose $\dl$ that will balance the two error terms above, i.e.\
$\dl$ that satisfies: $1-\errexp_{n}+\dl=1-\frac{\dl\lm_{n}}{n-\leng}$.
This $\dl$ is 
\[
\dl=\errexp_{n}\cdot\left(1-\frac{\lm_{n}}{n-\leng+\lm_{n}}\right)=\errexp_{n}\cdot\left(1-\frac{n^{2}}{2\left(n-\leng\right)\left(n^{2}-1\right)+n^{2}}\right).
\]
We conclude that in the case where  $\pairset$ is unbounded, then
\[
\#\brac{\sl n\left(\ZZ\right)\cap\byparby{\funddom}T\brac{\pairset}}=\frac{\mu\brac{\byparby{\funddom}T\brac{\pairset}}}{\mu\brac{\sl n\brac{\RR}/\sl n\brac{\ZZ}}}+O_{\pairset,\e}\brac{e^{\ind T\brac{1-\frac{\errexp_{n}n^{2}}{2\left(n-\leng\right)\left(n^{2}-1\right)+n^{2}}+\e}}}.
\]
By (\ref{eq: choose height}) and Proposition \ref{prop: volumes},
the latter equals 
\[
\frac{\vol_{\flagspace{\dvec}}\brac{\pairset}}{\prod_{j=1}^{\leng-1}\brac{d_{j}+d_{j+1}}}\cdot\frac{e^{\brac{2n-d_{1}-d_{\leng}}T}}{\mu\brac{\sl n\brac{\RR}/\sl n\brac{\ZZ}}}+O_{\pairset,\e}\brac{e^{\brac{2n-d_{1}-d_{\leng}}T\brac{1-\frac{\errexp_{n}n^{2}}{2\left(n-\leng\right)\left(n^{2}-1\right)+n^{2}}+\e}}}
\]
when $H=H_{\infty}$, and 
\[
\frac{\vol_{\flagspace{\dvec}}\brac{\pairset}}{\prod_{j=1}^{\leng-1}\brac{d_{j}+d_{j+1}}}\cdot e^{T}\sum_{i=0}^{\leng-2}(-1)^{\leng-2-i}\frac{T^{i}}{i!}+O_{\pairset,\e}\brac{e^{T\brac{1-\frac{\errexp_{n}n^{2}}{2\left(n-\leng\right)\left(n^{2}-1\right)+n^{2}}+\e}}}
\]
when $H=H_{\ac}$. When $\pairset$ is bounded, we use Proposition
\ref{thm: Counting with S and W}, and obtain of course the same main
term, but with an error term of $O_{\pairset,\e}\brac{e^{\ind T\left(1-\errexp+\e\right)}}$.

As for the leading constant, we recall that $\mu\brac{\sl n\brac{\RR}/\sl n\brac{\ZZ}}$
is the $\mu$-volume of a fundamental domain for $\sl n\brac{\ZZ}$,
which is given in (\ref{eq: Garrett}). All in all,
\[
\frac{\vol_{\flagspace{\dvec}}\brac{\pairset}}{\prod_{j=1}^{\leng-1}\brac{d_{j}+d_{j+1}}}\cdot\frac{1}{\prod_{i=2}^{n}\zeta\left(i\right)}=\frac{\vol_{\flagspace{\dvec}}\brac{\pairset}}{\prod_{j=1}^{\leng-1}\brac{d_{j}+d_{j+1}}}\cdot\frac{1}{\prod_{i=2}^{n}\zeta\left(i\right)}\cdot\frac{\mass{\vol_{\flagspace{\dvec}}}}{\mass{\vol_{\flagspace{\dvec}}}}=
\]
\[
=\frac{\vol_{\flagspace{\dvec}}^{1}\brac{\pairset}}{\prod_{j=1}^{\leng-1}\brac{d_{j}+d_{j+1}}}\frac{\mass{\vol_{\flagspace{\dvec}}}}{\prod_{i=2}^{n}\zeta\left(i\right)}=2^{\leng-1}\cdot c_{\dvec,n}\vol_{\flagspace{\dvec}}^{1}\brac{\pairset}.
\]
This completes the proof.
\end{proof}
\bibliographystyle{alpha}
\phantomsection\addcontentsline{toc}{section}{\refname}\bibliography{bib_for_flags}

\end{document}